\documentclass[reqno]{amsart}
\usepackage[T1]{fontenc}
\usepackage{lmodern}
\usepackage{amssymb,latexsym,amsmath,amsthm,enumerate,amsbsy}
\usepackage[mathscr]{eucal}
\usepackage{framed,color,graphicx}
\usepackage{array,enumitem,microtype,needspace}
\usepackage[hidelinks,hypertexnames=false]{hyperref}
\usepackage{mathrsfs}
\usepackage{cite}
\usepackage[all]{xy}
\usepackage{tikz}
\usetikzlibrary{shadows, shadings, calc}

\makeatletter
\@namedef{subjclassname@2020}{\textup{2020} Mathematics Subject Classification}
\makeatother

\numberwithin{equation}{section}

\usetikzlibrary{positioning,decorations.pathreplacing,patterns,decorations.pathmorphing}
\tikzset{%
element/.style={draw, shape=circle, fill=white, inner sep=1.4pt}
}

\DeclareSymbolFont{bbold}{U}{bbold}{m}{n}
\DeclareSymbolFontAlphabet{\mathbbold}{bbold}

\theoremstyle{plain}
\newtheorem{theorem}{Theorem}[section]
\newtheorem{lemma}[theorem]{Lemma}
\newtheorem{corollary}[theorem]{Corollary}
\newtheorem{proposition}[theorem]{Proposition}

\newtheorem{problem}[theorem]{Problem}

\theoremstyle{definition}

\newtheorem{remark}[theorem]{Remark}

\newcommand{\A}{\mathcal A_2^1}
\newcommand{\B}{\mathcal B_2^1}
\newcommand{\Abar}{\overline{\mathcal A_2^1}}

\renewcommand{\ge}{\geqslant}
\renewcommand{\le}{\leqslant}

\newcommand{\Mod}{\mathsf{Mod}}
\newcommand{\V}{\mathsf V}
\newcommand{\Id}{\operatorname{Id}}
\newcommand{\con}{\operatorname{con}}
\newcommand{\Finf}{\mathcal F_{\infty}}
\newcommand{\eps}{\varepsilon}

\newcommand{\ba}{\mathbf{a}}
\newcommand{\bb}{\mathbf{b}}

\newcommand{\bd}{\mathbf{d}}

\newcommand{\bp}{\mathbf{p}}
\newcommand{\bq}{\mathbf{q}}
\newcommand{\br}{\mathbf{r}}
\newcommand{\bs}{\mathbf{s}}
\newcommand{\bt}{\mathbf{t}}
\newcommand{\bu}{\mathbf{u}}
\newcommand{\bv}{\mathbf{v}}
\newcommand{\bw}{\mathbf{w}}
\newcommand{\bx}{\mathbf{x}}
\newcommand{\by}{\mathbf{y}}

\begin{document}
\title[Inherently nonfinitely based ai-semirings]
{Inherently nonfinitely based additively idempotent semirings}

\author{Miaomiao Ren}
\address{School of Mathematics, Northwest University, Xi'an, 710127, Shaanxi, P.R. China}
\email{miaomiaoren@yeah.net}

\author{Mengya Yue}
\address{School of Mathematics, Northwest University, Xi'an, 710127, Shaanxi, P.R. China}
\email{myayue@yeah.net}

\author{Yilin Zhou}
\address{School of Mathematics, Northwest University, Xi'an, 710127, Shaanxi, P.R. China}
\email{zhou990619@163.com}


\hypersetup{
  pdftitle={Inherently nonfinitely based additively idempotent semirings},
  pdfauthor={Miaomiao Ren, Mengya Yue, and Yilin Zhou},
  pdfsubject={Finite basis problems for additively idempotent semirings},
  pdfkeywords={additively idempotent semiring, finite basis problem, inherent nonfinite basability, Zimin word}
}

\begin{abstract}
We establish a sufficient condition for an additively idempotent
semiring to be inherently nonfinitely based: if its generated variety
is locally finite and every Zimin word is minimal in the additive
order, then it is contained in no finitely based locally finite
variety.  We characterize Zimin minimality by the membership of a
countable flat factor semiring $\Finf$ in the generated variety and
prove that $\V(\Finf)$ is a minimal inherently nonfinitely based
variety.  We also obtain general restrictions on the additive order of
finite inherently nonfinitely based ai-semirings.  As applications,
the six-element ai-semirings $\A$ and $\B$ are shown to be inherently
nonfinitely based.  Finally, we prove that
the six-element ai-semiring $\Abar$ is not inherently
nonfinitely based and is finitely based.  Its multiplicative reduct
is inherently nonfinitely based; thus, to the best of our knowledge,
$\Abar$ is the first finitely based ai-semiring with an inherently
nonfinitely based multiplicative reduct.
\end{abstract}

\maketitle

\section{Introduction and preliminaries}
\label{sec:introduction}

An \emph{additively idempotent semiring}, or an \emph{ai-semiring}, is
an algebra $(S,+,\cdot)$ such that $(S,+)$ is a commutative idempotent
semigroup, $(S,\cdot)$ is a semigroup, and multiplication distributes
over addition on both sides:
\[
  x(y+z)\approx xy+xz,
  \qquad
  (x+y)z\approx xz+yz.
\]
We work throughout in the signature consisting of the two binary
operations $+$ and $\cdot$, without designated constants.  Thus a
symbol $0$ occurring in a particular construction denotes an element
of the algebra, not a nullary operation.  General background on
semirings and semilattice-ordered semigroups can be found
in~\cite{Glazek2001,Golan1992,KurilPolak2005}.

The \emph{finite basis problem} asks whether the identities of a given
algebra admit a finite basis.  An algebra is \emph{locally finite} if
each of its finitely generated subalgebras is finite, and a variety is
locally finite if all of its members are locally finite.  A locally
finite variety is \emph{inherently nonfinitely based} if it is
contained in no finitely based locally finite variety of the same
signature.  An algebra is called inherently nonfinitely based if the
variety it generates has this property.  The notion originated in
semigroup theory; classical results of Perkins, Sapir, and Sapir and
Volkov form much of the relevant background
\cite{Perkins1969,Sapir1987a,SapirVolkov1988}.  For semirings, finite
nonfinitely based examples and broad classes of inherently
nonfinitely based algebras have been obtained in
\cite{Dolinka2007,Dolinka2009,JacksonRenZhao2022,
RenJacksonZhaoLei2023,Volkov2021,WuRenZhao2024,
DGV2025EndFiniteSemilattices}.

Every ai-semiring carries its natural additive order
\[
  a\leq_S b \quad\Longleftrightarrow\quad a+b=b.
\]
This order is compatible with multiplication, and $a+b$ is the join of
$a$ and $b$.  We therefore describe finite ai-semirings by their
additive Hasse diagrams whenever this is more convenient than giving
their additive tables.

The Zimin words are defined recursively by
\[
  Z_1=x_1,
  \qquad
  Z_{m+1}=Z_mx_{m+1}Z_m
  \quad(m\geq1).
\]
For nonempty words $\bu$ and $\bv$, write
\[
  S\models \bu\preceq \bv
  \quad\Longleftrightarrow\quad
  S\models \bu+\bv\approx \bv.
\]
We say that $\bv$ is \emph{minimal for $S$} if
$S\models \bu\preceq \bv$ implies that $\bu=\bv$ literally.  Minimality is
stronger than the isoterm property: it excludes not only nontrivial
identities involving $\bv$, but also nontrivial multiplicative words
lying below $\bv$ in the additive order.

The main part of the paper develops a general sufficient condition
based on minimal Zimin words.  For an algebra $S$, let $\V(S)$ denote
the variety generated by $S$, and let $\Id_{\leq n}(S)$ be the set of
identities of $S$ involving at most $n$ distinct variables in their
two sides together.  Our principal result is the following.

\begin{theorem}\label{thm:main}
Let $S$ be an ai-semiring for which every Zimin word is minimal.  For
each $n\geq1$, there exists an infinite ai-semiring $T_n$, generated by
at most $(6n+8)^2$ elements, such that
\[
  T_n\models\Id_{\leq n}(S).
\]
The algebra $T_n$ depends only on $n$, not on $S$.  If $\V(S)$ is
locally finite, then $S$ is inherently nonfinitely based.  In
particular, the final conclusion holds whenever $S$ is finite.
\end{theorem}

The algebras $T_n$ are flat extensions of factor semigroups arising
from Sapir's construction and used by Auinger, Dolinka, and
Volkov~\cite{ADV2012}.  The proof transfers all identities with a
bounded number of variables from $S$ to a fixed infinite finitely
generated flat semiring.  The argument applies to arbitrary semiring
identities, including nonregular identities in which a variable may
occur on only one side.

There is also a single countable flat semiring that records the
minimality of all Zimin words.  Let $Z_\infty$ be the right-infinite
word whose prefixes are the Zimin words, let $L_\infty$ be the set of
its nonempty factors, and let $\Finf$ be the associated flat factor
semiring.

\begin{theorem}\label{thm:characterization}
For every ai-semiring $S$, the following conditions are equivalent:
\begin{enumerate}[label=\textup{(\roman*)}]
\item every Zimin word is minimal for $S$;
\item $\Finf\in\V(S)$.
\end{enumerate}
Moreover, $\V(\Finf)$ is locally finite and inherently nonfinitely
based.  If $\B$ denotes the six-element Brandt ai-semiring, then
\[
  \B\notin\V(\Finf).
\]
\end{theorem}

The conclusion about $\V(\Finf)$ can be sharpened.  We prove that it
is a minimal inherently nonfinitely based variety: every proper
subvariety is contained in a finitely based, locally finite variety.
The containing variety can be defined by the ai-semiring identities
together with four multiplicative identities.

We also derive restrictions on the additive order of a finite
inherently nonfinitely based ai-semiring.  Its additive order can be
neither a chain nor a poset of height at most two; consequently, its
width is at least two and its height is at least three.  If it has a
multiplicative identity, that element is neither the greatest nor the
least element of the additive order, and every multiplicative subgroup
is an antichain.

The two six-element ai-semirings $\A$ and $\Abar$ should
not be confused.  They have the same multiplicative reduct $A_2^1$
but different additive orders.  We prove that $\A$ is inherently
nonfinitely based, whereas $\Abar$ is not.  Together with
the Brandt ai-semiring $\B$, these examples show that the additive
structure can decisively affect inherent nonfinite basability even
when the multiplication is fixed.  The result for
$\Abar$ answers a question of Dolinka.  To the best of our
knowledge, it is the first counterexample to the converse of the
general implication from inherent nonfinite basability of an
ai-semiring to that of its multiplicative reduct, and it is the
smallest possible counterexample.  Since $\Abar$ is finitely based,
it also provides, to the best of our knowledge, the first finitely
based ai-semiring whose multiplicative reduct is inherently
nonfinitely based.

We then determine the ordinary finite basis status of $\Abar$.
Noninherent nonfinite basability alone does not settle this question.
Our finite basis contains at most 109 identities. Its completeness
proof starts with an arbitrary valid absorption $\bu\approx \bu+\bq$.
After deleting irrelevant summands, we first derive
$\bq\leq \bu+\bq^2$ and then absorb $\bq^2$ into $\bu$. A content-contraction
lemma reduces the coordinate identities to seventeen fixed laws.
The main argument is given in Section~\ref{sec:finite-basis};
the auxiliary finite families and their derivations are collected
in the appendix.

The paper is organized into four sections.  The present section gives
the introduction and the preliminary facts about flat factor
semirings.  Section~\ref{sec:infb} proves the general inherent
nonfinite basis criterion, characterizes Zimin minimality through
$\Finf$, proves that $\V(\Finf)$ is a minimal inherently nonfinitely
based variety, establishes additive-order restrictions for finite
INFB ai-semirings, and applies the criterion to $\A$ and $\B$.
Section~\ref{sec:not-infb} proves that $\Abar$ is not
inherently nonfinitely based.  Section~\ref{sec:finite-basis} then
solves its ordinary finite basis problem. The appendix contains the
finite identity basis and the auxiliary proofs used in that section.

\paragraph{Preliminaries and flat factor semirings.}\label{sec:flat}

We first fix the notation used in the proof. We refer
to~\cite{BurrisSankappanavar1981} for the standard facts from universal
algebra. Let $X$ be a countably infinite alphabet, let $X^+$ be the free
semigroup of nonempty words over $X$, and put
$X^*=X^+\cup\{\varepsilon\}$. A factor always means a contiguous factor.
Word substitutions are nonerasing; the empty word is used only as a
context. We use ordinary italic letters for individual variables and
algebra elements, and bold letters for generic words, word contexts,
polynomials, and formal terms. Named auxiliary expressions and operators
retain their indicated notation.

A language $L\subseteq A^+$ is \emph{factorial} if it contains every nonempty factor of each of its words. On $L\cup\{0\}$ define
\begin{equation}\label{eq:flat}
p\cdot q=
\begin{cases}
pq,&p,q\in L\text{ and }pq\in L,\\
0,&\text{otherwise},
\end{cases}
\qquad
p+q=
\begin{cases}
p,&p=q,\\
0,&p\ne q.
\end{cases}
\end{equation}
Denote the resulting algebra by $F(L)$. In particular, $0$ absorbs both operations.

\begin{lemma}\label{lem:flat}
The algebra $F(L)$ is an ai-semiring. A product of nonzero elements is their literal concatenation if that concatenation belongs to $L$, and is zero otherwise. A finite nonempty sum is nonzero if and only if all its summands have the same nonzero value.
\end{lemma}
\begin{proof}
Factoriality shows that both bracketings of a triple product give its concatenation when that concatenation belongs to $L$, and give zero otherwise. This proves associativity and, by induction, the assertion about products. Flat addition is a semilattice operation: $0$ is its greatest element and all other elements are pairwise incomparable.

For left distributivity, only the case $p\ne q$ needs consideration. If $sp=sq\ne0$, then $s,p,q$ are words and left cancellation of their literal concatenations gives $p=q$, a contradiction. Thus $sp+sq=0=s(p+q)$. Right distributivity follows by right cancellation. The assertion about sums follows immediately by induction.
\end{proof}

\begin{lemma}\label{lem:normal}
Every term in an ai-semiring is equivalent, using the ai-semiring axioms, to a finite nonempty sum of nonempty words. Every nonempty factor of a minimal word is minimal.
\end{lemma}
\begin{proof}
The first assertion follows by distributing products over sums and using associativity. For the second, write $\mathbf{W}=\boldsymbol{\alpha} \bt\boldsymbol{\beta}$, where $\mathbf{W}$ is minimal and $\bt$ is nonempty. If $S\models \bs\preceq \bt$, distributivity and monotonicity under multiplication give
\[
S\models\boldsymbol{\alpha} \bs\boldsymbol{\beta}\preceq\boldsymbol{\alpha} \bt\boldsymbol{\beta}=\mathbf{W}.
\]
Multiplication is omitted on an empty context side. Minimality gives the literal equality $\boldsymbol{\alpha} \bs\boldsymbol{\beta}=\boldsymbol{\alpha} \bt\boldsymbol{\beta}$. Cancellation in the free monoid gives $\bs=\bt$.
\end{proof}

We also need the following elementary coding fact.

\begin{lemma}\label{lem:code}
Let $C$ be a finite nonempty subset of a factorial language $L$. There is a finite nonempty prefix code $D\subseteq L$ such that $|D|\leq|C|$ and $C\subseteq D^+$. If $D=\{\bd_1,\ldots,\bd_\ell\}$, the morphism
\[
f:\{y_1,\ldots,y_\ell\}^+\longrightarrow A^+,
\qquad f(y_i)=\bd_i,
\]
is injective.
\end{lemma}
\begin{proof}
Start with the set $C$. Whenever two distinct current members are $\bu$ and $\bu\bv$ with $\bv\ne\eps$, replace $\bu\bv$ by $\bv$. The current set remains nonempty and its cardinality does not increase. Each former generator is a product of new generators, so $C\subseteq D^+$ is preserved. All new words are factors of old words and remain in $L$. The sum of the lengths of the current members strictly decreases. The procedure therefore terminates at a prefix code.

For injectivity, compare two equal concatenations of codewords. Their first codewords are prefix-comparable and hence equal. Cancel them and repeat. Nonemptiness of codewords ensures that both concatenations end at the same step, proving equality of the original words over the $y_i$.
\end{proof}

\section{Inherently nonfinitely based ai-semirings}\label{sec:infb}

We now prove the inherent nonfinite basis results stated in the
introduction.  The first step is to encode Zimin minimality in a
single flat factor semiring.

\paragraph{The Zimin factor semiring.}\label{sec:ziminfactor}

Define $\Finf=F(L_\infty)$ using \eqref{eq:flat}. Its multiplicative reduct is the Zimin factor semigroup used by Sapir \cite{Sapir1987a}. It is countably infinite, since it contains all the distinct one-letter words $x_i$.

\begin{proposition}\label{prop:membership}
For any ai-semiring $S$, every Zimin word is minimal for $S$ if and only if $\Finf\in\V(S)$.
\end{proposition}
\begin{proof}
First, all Zimin words are minimal for $\Finf$. Indeed, if $\Finf\models \bu+Z_m\approx Z_m$, assign each variable the one-letter word bearing its name, renaming any additional variables as necessary. The value of $Z_m$ is the nonzero word $Z_m$. Flat addition forces the value of $\bu$ to be exactly this word. Nonzero evaluation is literal concatenation, so $\bu=Z_m$. It follows at once that $\Finf\in\V(S)$ implies Zimin minimality for $S$.

Conversely, suppose all Zimin words are minimal for $S$. Take any identity of $S$ and normalize it to
\[
\mathbf{P}=\bu_1+\cdots+\bu_p\approx \bv_1+\cdots+\bv_q=\mathbf{Q}.
\]
Consider an assignment $\varphi$ into $\Finf$. If both sides are zero, they agree. Otherwise interchange the sides if necessary and assume $\varphi(\mathbf{P})=w\ne0$. Then each $\bu_i$ evaluates to the literal word $w$, and every variable in $\mathbf{P}$ has nonzero value.

For each nonzero-valued variable, substitute its value as a word. For each zero-valued variable, substitute a fresh letter absent from all these nonzero values. This defines a nonerasing word substitution $\theta$, and $\theta(\bu_i)=w$ literally for every $i$. Substitution in the original identity therefore gives
\[
S\models w\approx\sum_{j=1}^q\theta(\bv_j),
\qquad S\models\theta(\bv_j)\preceq w\quad(1\leq j\leq q).
\]
Every finite factor of $Z_\infty$ is a factor of some $Z_m$, so $w$ is minimal for $S$ by Lemma~\ref{lem:normal}. Thus $\theta(\bv_j)=w$ literally. In particular, none of these words contains a fresh letter used for a zero-valued variable. Returning to $\varphi$, each $\bv_j$ has value $w$, whence $\varphi(\mathbf{Q})=w$.

We have proved $\Finf\models\Id(S)$. Birkhoff's theorem yields $\Finf\in\V(S)$.
\end{proof}

\begin{lemma}\label{lem:largest}
The letter with largest index in a nonempty finite factor of $Z_\infty$ occurs exactly once. In particular, $Z_\infty$ has no nonempty square as a factor.
\end{lemma}
\begin{proof}
The first assertion follows inductively for factors of $Z_m=Z_{m-1}x_mZ_{m-1}$. A factor containing $x_m$ contains it once, and a factor avoiding $x_m$ lies in one of the two copies of $Z_{m-1}$. Every finite factor of $Z_\infty$ lies in some $Z_m$. A square would contain its largest-index letter at least twice, proving the second assertion.
\end{proof}

\begin{proposition}\label{prop:local}
The variety $\V(\Finf)$ is locally finite. More precisely, for $n\geq1$, put
\[
M_n=\sum_{j=1}^{2^n-1}n^j.
\]
Then its free algebra on $n$ generators has at most $2^{M_n}$ elements. Moreover,
\begin{equation}\label{eq:squares}
\Finf\models x^2\approx y^2.
\end{equation}
\end{proposition}
\begin{proof}
We first show that a word in at most $n$ variables with a nonzero evaluation in $\Finf$ has length at most $2^n-1$. Regard the nonzero variable values as nonempty words. Their concatenation belongs to $L_\infty$. Its largest-index letter occurs once by Lemma~\ref{lem:largest}. A variable whose image contains this letter must therefore occur exactly once in the original word. Write the word as $\bu y\bv$, where $y$ is this variable. Each nonempty side uses at most $n-1$ variables and has nonzero evaluation. Induction, allowing an empty side to have length zero, gives
\[
|\bu y\bv|\leq(2^{n-1}-1)+1+(2^{n-1}-1)=2^n-1.
\]
The base case is included by taking the bound for zero variables to be zero.

Consequently, every word of length at least $2^n$ in $n$ variables represents the zero term function. Normalize an arbitrary term by Lemma~\ref{lem:normal}. If it has a long summand, it is identically zero because zero absorbs addition. Otherwise it is a nonempty sum drawn from the $M_n$ short words. Additive commutativity and idempotence give at most $2^{M_n}-1$ such sums, together with the zero function. Elements of the relatively free algebra are term functions modulo the identities of $\Finf$, proving the stated bound.

Finally, square-freeness implies $a^2=0$ for every $a\in\Finf$, including $a=0$. This proves the constant-free identity \eqref{eq:squares}.
\end{proof}

\paragraph{Sapir's factor languages and guarded identities.}\label{sec:sapir}

Fix $k\geq1$, put $r=6k+2$, and let
\[
A_k=\{a_{ij}:1\leq i,j\leq r\}.
\]
Define a morphism $\gamma:A_k^+\to A_k^+$ by
\begin{equation}\label{eq:gamma}
\gamma(a_{ij})=a_{i1}a_{j2}a_{i3}a_{j4}\cdots a_{i,r-1}a_{jr}.
\end{equation}
Let $L_k$ be the nonempty factors of the words $\gamma^m(a_{11})$, for $m\geq1$, and set
\[
F_k=F(L_k),\qquad V_k^0=(L_k\cup\{0\},\cdot).
\]
This is the construction in \cite[p.~947]{ADV2012}. We use the following combinatorial theorem in exactly its ordinary semigroup form.

\begin{proposition}[Sapir; Auinger--Dolinka--Volkov]\label{prop:sapir}
Let $\bs,\bt$ be nonempty words using at most $k$ variables in total. If $V_k^0$ fails $\bs\approx \bt$, then this single identity entails, in the class of semigroups, an identity $Z_{k+1}\approx \bq$ with $\bq\ne Z_{k+1}$ literally.
\end{proposition}

This is \cite[Proposition~2.5]{ADV2012}, where it is attributed to Sapir. No semiring entailment is being asserted here.

\begin{lemma}\label{lem:return}
Every letter of $A_k$ occurs in $\gamma^3(c)$ for every $c\in A_k$. For every $\bw\in L_k$ there exists a nonempty $\bd\in L_k$ such that $\bw\bd\bw\in L_k$. The ai-semiring $F_k$ is infinite and is generated by its $r^2$ letters.
\end{lemma}
\begin{proof}
For each $p$, the word $\gamma(c)$ contains a letter with second index $p$. Its image contains $a_{pq}$ for every even $q$. Thus $\gamma^2(c)$ contains all even-column letters, in particular every $a_{p2}$. Their images contain all odd-column letters. Also, choosing any letter $e$ in $\gamma(c)$ shows that $\gamma^3(c)$ contains the factor $\gamma^2(e)$ and hence all even-column letters as well.

Choose $m\geq1$ such that $\bw$ is a factor of $\gamma^m(a_{11})$. Since $a_{11}$ occurs in $\gamma^3(c)$ for every letter $c$, the word $\gamma^{m+3}(c)$ contains $\bw$. Write $\gamma(a_{11})=c_1\cdots c_r$. In
\[
\gamma^{m+4}(a_{11})=\gamma^{m+3}(c_1)\cdots\gamma^{m+3}(c_r),
\]
select one occurrence of $\bw$ in the first block and another in the third. The intervening word $\bd$ contains the entire second block and is nonempty. The spanning factor is $\bw\bd\bw$; factoriality also gives $\bd\in L_k$.

All letters belong to $L_k$. They generate every word by concatenation, and two distinct letters have sum zero. The words $\gamma^m(a_{11})$ have unbounded lengths $r^m$, so $F_k$ is infinite. Lemma~\ref{lem:flat} proves the ai-semiring assertion.
\end{proof}

\begin{lemma}[Guarded identity]\label{lem:guard}
Suppose all Zimin words are minimal for $S$ and
\[
S\models \bs\preceq \bt,\qquad |\con(\bs)\cup\con(\bt)|\leq k-1.
\]
If $z$ is a fresh variable, then $V_k^0\models \bt z\bt\approx \bs z\bt$.
\end{lemma}
\begin{proof}
Let $\eta$ be the semigroup identity $\bt z\bt\approx \bs z\bt$. We claim that no nontrivial Zimin identity follows from $\eta$. Semigroup equational consequence can be represented by finite sequences of elementary replacements
\[
\boldsymbol{\alpha}\theta(\bt z\bt)\boldsymbol{\beta}\longleftrightarrow\boldsymbol{\alpha}\theta(\bs z\bt)\boldsymbol{\beta},
\qquad \boldsymbol{\alpha},\boldsymbol{\beta}\in X^*,
\]
with $\theta$ nonerasing. To recall why, the relation defined by these finite sequences is a congruence on the free semigroup, and its quotient satisfies $\eta$. Thus every semigroup consequence of $\eta$ identifies words connected by such a sequence.

Consider any elementary replacement starting at a literal Zimin word $Z_m$. In either direction, the replaced factor contains the suffix $\theta(\bt)$. This nonempty word is a factor of $Z_m$, hence is minimal for $S$ by Lemma~\ref{lem:normal}. Substitution gives $S\models\theta(\bs)\preceq\theta(\bt)$, so $\theta(\bs)=\theta(\bt)$ literally. The replacement therefore changes nothing. Induction along a replacement sequence shows that $Z_m$ cannot be changed to a different word.

The identity $\eta$ has at most $k$ variables. If it failed in $V_k^0$, Proposition~\ref{prop:sapir} would contradict the claim. Hence it holds in $V_k^0$.
\end{proof}

\begin{lemma}[Transfer with zero values allowed]\label{lem:transfer}
Let $\nu$ be an assignment into $F_k$.  Under the hypotheses of
Lemma~\ref{lem:guard},
\[
\nu(\bt)\ne0\quad\Longrightarrow\quad\nu(\bs)=\nu(\bt).
\]
No restriction is imposed on the values of variables outside $\bt$.
\end{lemma}
\begin{proof}
Put $w=\nu(\bt)\ne0$. Choose $d\in L_k$ with $wdw\in L_k$ by Lemma~\ref{lem:return}, and assign $z$ to $d$. Lemma~\ref{lem:guard} yields
\[
wdw=\nu(\bs)dw.
\]
The left side is nonzero, so the right side is nonzero too, and $\nu(\bs)$ is a word in $L_k$. Both products are literal concatenations. Right cancellation of the common suffix $dw$ in the free monoid gives $\nu(\bs)=w$.
\end{proof}

\paragraph{Transfer of identities with bounded variable number.}\label{sec:transfer}

\begin{theorem}\label{thm:bounded}
If all Zimin words are minimal for $S$, then
\[
F_{n+1}\models\Id_{\leq n}(S)\qquad(n\geq1).
\]
\end{theorem}
\begin{proof}
Put $k=n+1$, and take an identity of $S$ in at most $n$ variables. Normalize it as
\[
\mathbf{P}=\bu_1+\cdots+\bu_p\approx \bv_1+\cdots+\bv_q=\mathbf{Q},
\qquad p,q\geq1.
\]
Let $X$ be the union of its variable sets. Fix an arbitrary assignment $\varphi:X\to F_k$. If both sides evaluate to zero there is nothing to prove. Otherwise, by symmetry, suppose $\varphi(\mathbf{P})=w\ne0$. Lemma~\ref{lem:flat} gives
\begin{equation}\label{eq:common}
\varphi(\bu_1)=\cdots=\varphi(\bu_p)=w.
\end{equation}

Partition $X$ as
\[
X_+=\{x\in X:\varphi(x)\ne0\},\qquad X_0=X\setminus X_+.
\]
All variables in the $\bu_i$ belong to $X_+$. We do not assume that this holds for the $\bv_j$. The set
\[
C=\{\varphi(x):x\in X_+\}\subseteq L_k
\]
is nonempty. By Lemma~\ref{lem:code}, choose a prefix code $D=\{\bd_1,\ldots,\bd_\ell\}\subseteq L_k$ such that $C\subseteq D^+$ and $\ell\leq|X_+|$. Let $Y=\{y_1,\ldots,y_\ell\}$, and let $f:Y^+\to A_k^+$ be the injective morphism sending $y_i$ to $\bd_i$.

Choose a nonempty codeword decomposition of each $\varphi(x)$ for $x\in X_+$, and let $\sigma(x)\in Y^+$ be the corresponding word of code symbols. For every $x\in X_0$, introduce a separate fresh variable $z_x$ and put $\sigma(x)=z_x$. This is a nonerasing substitution, with at most
\begin{equation}\label{eq:variablebound}
\ell+|X_0|\leq|X_+|+|X_0|=|X|\leq n=k-1
\end{equation}
variables in its images.

For each $i$, nonzero evaluation in \eqref{eq:common} gives $f(\sigma(\bu_i))=w$ as a literal word. Injectivity of $f$ implies
\[
\sigma(\bu_1)=\cdots=\sigma(\bu_p)=\bt
\]
literally. Put $\bs_j=\sigma(\bv_j)$. Substitution in the valid identity, followed by additive idempotence, gives
\[
S\models \bt\approx \bs_1+\cdots+\bs_q,
\quad\text{and hence}\quad S\models \bs_j\preceq \bt\quad(1\leq j\leq q).
\]
Now define an assignment $\nu$ into $F_k$ by
\[
\nu(y_i)=\bd_i,\qquad \nu(z_x)=0.
\]
Its composition with $\sigma$ agrees with $\varphi$ on every original variable. In particular, $\nu(\bt)=w\ne0$. The variable bound \eqref{eq:variablebound} allows us to apply Lemma~\ref{lem:transfer} to every $\bs_j\preceq \bt$, including those for which $\bs_j$ might initially have contained a $z_x$. We obtain
\[
\varphi(\bv_j)=\nu(\bs_j)=w\quad(1\leq j\leq q).
\]
Thus $\varphi(\mathbf{Q})=w=\varphi(\mathbf{P})$. Every assignment has been covered, proving the theorem.
\end{proof}

\begin{proof}[Proof of Theorem~\ref{thm:main}]
Take $T_n=F_{n+1}$. Lemma~\ref{lem:return} makes it infinite and generated by
\[
\bigl(6(n+1)+2\bigr)^2=(6n+8)^2
\]
letters, and Theorem~\ref{thm:bounded} proves the required identities.

Suppose a finitely based locally finite variety $\mathcal U$ of signature $\{+,\cdot\}$ contains $S$. Let $\Sigma$ be a finite basis for $\mathcal U$, and choose $n\geq1$ bounding the number of variables in each identity of $\Sigma$. Since $S\in\mathcal U$, all of $\Sigma$ belongs to $\Id_{\leq n}(S)$. Hence $T_n\models\Sigma$ and $T_n\in\mathcal U$, contradicting local finiteness. If $\V(S)$ is locally finite, this proves that $S$ is INFB. Every finite algebra generates a locally finite variety, giving the finite case.
\end{proof}

\begin{remark}
The proof does not assume in advance that the supervariety $\mathcal U$ consists of ai-semirings. It only uses the fact that its finite basis consists of identities valid in $S$, all of which transfer to the ai-semiring $T_n$.
\end{remark}

\paragraph{Completion of the characterization theorem.}\label{sec:consequences}

We first complete the proof of the characterization stated in the
introduction.

\begin{proof}[Proof of Theorem~\ref{thm:characterization}]
The equivalence follows from Proposition~\ref{prop:membership}.  By
Proposition~\ref{prop:local}, the variety $\V(\Finf)$ is locally finite.
The proof of Proposition~\ref{prop:membership} also shows that every
Zimin word is minimal for $\Finf$.  Hence Theorem~\ref{thm:main}
implies that $\V(\Finf)$ is inherently nonfinitely based.

Identity~\eqref{eq:squares} holds in $\Finf$ but fails in $\B$:
assign $x$ to the identity matrix and $y$ to the matrix unit $E_{12}$.
Their squares are respectively the identity matrix and the zero matrix.
Therefore $\B\notin\V(\Finf)$.
\end{proof}

\begin{corollary}\label{cor:least}
The variety $\V(\Finf)$ is the least ai-semiring variety in which every
Zimin word is minimal.  It is locally finite and inherently
nonfinitely based.
\end{corollary}

\begin{proof}
The first assertion follows from Proposition~\ref{prop:membership},
applied to every member of the variety.  The remaining assertions
follow from Theorem~\ref{thm:characterization}.
\end{proof}

\paragraph{Minimality of $\V(\Finf)$.}

Corollary~\ref{cor:least} identifies $\V(\Finf)$ as the least variety
in which every Zimin word is minimal.  We now prove the stronger fact
that it is minimal among inherently nonfinitely based varieties.  The
following identity provides the link between additive minimality and
multiplicative word identities:
\begin{equation}\label{eq:flat-bridge}
  (x+y)z(x+y)\approx xz(x+y).
\end{equation}

\begin{lemma}\label{lem:flat-bridge}
Every flat ai-semiring satisfies \eqref{eq:flat-bridge}.  Consequently,
the identity holds in every variety generated by flat ai-semirings.
\end{lemma}

\begin{proof}
Let $R$ be a flat ai-semiring, with multiplicative zero $0$, and take
$a,b,c\in R$.  If $a=b$, then
\[
  (a+b)c(a+b)=aca=ac(a+b).
\]
If $a\ne b$, then $a+b=0$, and both sides of
\eqref{eq:flat-bridge} evaluate to $0$.  Thus the identity holds in
$R$, and hence in every variety generated by flat ai-semirings.
\end{proof}

\begin{lemma}\label{lem:inequality-to-word-identity}
Let $S$ be an ai-semiring satisfying \eqref{eq:flat-bridge}.  Suppose
that, for some nonempty word $\mathbf{W}\ne Z_n$,
\begin{equation}\label{eq:nonminimal-zimin}
  S\models \mathbf{W}+Z_n\approx Z_n.
\end{equation}
If $z$ occurs in neither $\mathbf{W}$ nor $Z_n$, then
\begin{equation}\label{eq:derived-zimin-identity}
  S\models Z_nzZ_n\approx \mathbf{W}zZ_n.
\end{equation}
Moreover, \eqref{eq:derived-zimin-identity} is a nontrivial identity of
multiplicative words, and its left side is a renaming of $Z_{n+1}$.
\end{lemma}

\begin{proof}
Substitute $\mathbf{W}$ for $x$ and $Z_n$ for $y$ in
\eqref{eq:flat-bridge}.  Using \eqref{eq:nonminimal-zimin}, we obtain
\[
  Z_nzZ_n
  \approx (\mathbf{W}+Z_n)z(\mathbf{W}+Z_n)
  \approx \mathbf{W}z(\mathbf{W}+Z_n)
  \approx \mathbf{W}zZ_n.
\]
If the two sides were literally the same word, cancellation of the
common suffix $zZ_n$ in the free semigroup would give $\mathbf{W}=Z_n$, a
contradiction.  The last assertion follows directly from the recursive
definition of the Zimin words.
\end{proof}

\begin{theorem}\label{thm:minimal-infb-variety}
Let
\[
  \mathcal N=\V(\Finf).
\]
Then $\mathcal N$ is a minimal inherently nonfinitely based variety of
ai-semirings.  More precisely, every proper subvariety of $\mathcal N$
is contained in a locally finite variety axiomatized by the
ai-semiring identities together with four multiplicative identities.
\end{theorem}

\begin{proof}
Theorem~\ref{thm:characterization} shows that $\mathcal N$ is locally
finite and inherently nonfinitely based.  Let
$\mathcal U\subsetneq\mathcal N$, and let $F$ be the countably
generated free algebra of $\mathcal U$.  If every Zimin word were
minimal for $F$, Proposition~\ref{prop:membership} would give
$\Finf\in\V(F)=\mathcal U$, whence
$\mathcal N\subseteq\mathcal U$, a contradiction.  Thus there are
$n\geq1$ and a nonempty word $\mathbf{W}\ne Z_n$ such that
\begin{equation}\label{eq:proper-nonminimal}
  \mathcal U\models \mathbf{W}+Z_n\approx Z_n.
\end{equation}

Since $\Finf$ is flat, Lemma~\ref{lem:flat-bridge} implies that
$\mathcal N$, and therefore $\mathcal U$, satisfies
\eqref{eq:flat-bridge}.  Lemma~\ref{lem:inequality-to-word-identity}
now yields the nontrivial multiplicative identity
\begin{equation}\label{eq:proper-zimin-identity}
  \mathcal U\models Z_nzZ_n\approx \mathbf{W}zZ_n,
\end{equation}
where $z$ is a new variable.

Every square in $\Finf$ is the multiplicative zero.  Hence the
following constant-free identities hold in $\Finf$ and throughout
$\mathcal U$:
\begin{equation}\label{eq:proper-square-identities}
  x^2\approx y^2,
  \qquad x^2y\approx x^2,
  \qquad yx^2\approx x^2.
\end{equation}
Define
\begin{equation}\label{eq:proper-overvariety}
  \mathcal E_{\mathbf{W}}=
  \Mod\bigl(\mathsf{AI}\cup
  \{x^2\approx y^2,\ x^2y\approx x^2,\ yx^2\approx x^2,
    \ Z_nzZ_n\approx \mathbf{W}zZ_n\}\bigr),
\end{equation}
where $\mathsf{AI}$ denotes the six standard ai-semiring identities.
Equations~\eqref{eq:proper-zimin-identity} and
\eqref{eq:proper-square-identities} give
$\mathcal U\subseteq\mathcal E_{\mathbf{W}}$, and $\mathcal E_{\mathbf{W}}$ is finitely
based by construction.

It remains to prove that $\mathcal E_{\mathbf{W}}$ is locally finite.  Let
$R\in\mathcal E_{\mathbf{W}}$.  The identity $x^2\approx y^2$ implies that all
squares in $R$ have a common value $c$.  The other two identities in
\eqref{eq:proper-square-identities} give
\[
  cr=rc=c\quad(r\in R),
  \qquad r^2=c\quad(r\in R).
\]
Thus $c$ is a multiplicative zero and the multiplicative reduct of
$R$ is a nil-semigroup.  It also satisfies the nontrivial Zimin word
identity \eqref{eq:proper-zimin-identity}.  By
\cite[Lemma~3.3.9]{Sapir2014}, every finitely generated nil-semigroup
with this property is finite.

Now let $A$ be a finite subset of $R$, and let
$M=\langle A\rangle_{\cdot}$.  The preceding paragraph shows that
$M$ is finite.  By distributivity, every element of the subsemiring
generated by $A$ is a finite nonempty sum of elements of $M$.
Addition is commutative and idempotent, so there are at most
$2^{|M|}-1$ such sums.  Hence every finitely generated member of
$\mathcal E_{\mathbf{W}}$ is finite, and $\mathcal E_{\mathbf{W}}$ is locally finite.

We have therefore embedded the arbitrary proper subvariety
$\mathcal U$ in a finitely based, locally finite variety.  Thus
$\mathcal U$ is not inherently nonfinitely based, and $\mathcal N$ is
a minimal inherently nonfinitely based variety.
\end{proof}

\begin{remark}\label{rem:minimal-not-limit}
Theorem~\ref{thm:minimal-infb-variety} does not assert that every
proper subvariety of $\mathcal N$ is itself finitely based; it only
places each such subvariety inside a finitely based, locally finite
variety.  Thus the theorem does not show that $\mathcal N$ is a limit
variety.  Nor does minimality imply that $\mathcal N$ is contained in
every inherently nonfinitely based ai-semiring variety.
\end{remark}

\paragraph{Restrictions imposed by the additive order.}
For the rest of this paragraph, write
\[
  a\leq b\quad\Longleftrightarrow\quad a+b=b
\]
for the natural additive order.  Distributivity implies that
multiplication preserves this order on both sides.  We first record
the semigroup-theoretic facts used below.

\begin{lemma}\label{lem:infb-reduct}
Let $S$ be a finite ai-semiring, and let $S_\cdot$ denote its
multiplicative reduct.  If $S$ is INFB, then $S_\cdot$ is INFB.
\end{lemma}
\begin{proof}
This is \cite[Theorem~2.11(3)]{JacksonRenZhao2022}.  For completeness,
suppose that $S_\cdot$ belongs to a finitely based locally finite
semigroup variety with finite basis $\Sigma$.  The identities in
$\Sigma$, together with the ai-semiring identities, define a finitely
based variety containing $S$.  This variety is locally finite: a
finite set of generators produces only finitely many multiplicative
words modulo $\Sigma$, and every semiring term is equivalent, by
distributivity, to a nonempty sum of such words.  Since addition is
commutative and idempotent, only finitely many sums occur.  Thus $S$
is not INFB, proving the contrapositive.
\end{proof}

We shall also use the following consequences of Sapir's criterion
\cite[Theorem~1(A),(B)]{Sapir1987a}.

\begin{lemma}\label{lem:sapir-order-test}
The following statements hold.
\begin{enumerate}[label=\textup{(\arabic*)}]
\item Every finite INFB semigroup contains an INFB subsemigroup with
an identity element.
\item Let $M$ be a finite monoid, let $d$ be the least common multiple
of the exponents of its subgroups, and let $H_f$ be the maximal
subgroup with identity $f$.  If
\begin{equation}\label{eq:order-sapir-test}
  faf=fa^{d+1}f\in H_f
\end{equation}
whenever $f$ is idempotent and $a$ divides $f$, then $M$ is not INFB.
\end{enumerate}
Here $a$ \emph{divides} $f$ if $f=uav$ for some $u,v\in M$.
\end{lemma}

\begin{lemma}\label{lem:local-monoids}
Let $T$ be a finite semigroup.  If $eTe$ is not INFB for every
idempotent $e\in T$, then $T$ is not INFB.
\end{lemma}
\begin{proof}
Otherwise Lemma~\ref{lem:sapir-order-test}(1) gives an INFB
subsemigroup $M$ of $T$ with identity $e$.  Since $M\subseteq eTe$,
the semigroup $eTe$ would also be INFB, a contradiction.
\end{proof}

The position of a multiplicative identity in the additive order is
already severely restricted.

\begin{proposition}\label{prop:extreme-identity}
Let $S$ be a finite ai-semiring with multiplicative identity $1$.  If
$1$ is either the greatest or the least element of the additive
order, then $S$ is not INFB.  Consequently, if $S$ is INFB, then
there exist $a,b\in S$ such that
\[
  a\nleq 1\qquad\text{and}\qquad 1\nleq b.
\]
\end{proposition}
\begin{proof}
Put $M=S_\cdot$, and use the notation of
Lemma~\ref{lem:sapir-order-test}(2).  Let $f$ be idempotent and write
$f=uav$.  If $1$ is greatest, then
\[
  f=uav\leq 1a1=a\leq1,
\]
and hence, for every $k\geq1$,
\[
  f=f^{k+2}\leq fa^kf\leq f1f=f.
\]
If $1$ is least, then $1\leq a\leq f$, and the inequalities are
reversed:
\[
  f=f1f\leq fa^kf\leq f^{k+2}=f.
\]
Thus $fa^kf=f$ in either case.  In particular,
\[
  faf=fa^{d+1}f=f\in H_f.
\]
Lemma~\ref{lem:sapir-order-test}(2) shows that $M$ is not INFB, and
Lemma~\ref{lem:infb-reduct} completes the proof.
\end{proof}

Multiplicative subgroups are also visible in the additive order.

\begin{proposition}\label{prop:subgroups-antichain}
Every subgroup of the multiplicative reduct of a finite ai-semiring
is an antichain in the additive order.
\end{proposition}
\begin{proof}
Let $G$ be such a subgroup, with identity $e$, and suppose that
$g\leq h$ in $G$.  Multiplication on the left by $g^{-1}$ gives
$e\leq r$, where $r=g^{-1}h$.  Since $G$ is finite, $r^m=e$ for some
$m>0$, and therefore
\[
  e\leq r\leq r^2\leq\cdots\leq r^m=e.
\]
It follows that $r=e$ and hence $g=h$.
\end{proof}

\begin{corollary}\label{cor:subgroup-sums}
If $G$ is a multiplicative subgroup and $g,h\in G$ are distinct,
then $g+h\notin G$.
\end{corollary}
\begin{proof}
Otherwise $g\leq g+h$ and $h\leq g+h$ are comparisons between
elements of $G$, contradicting
Proposition~\ref{prop:subgroups-antichain}.
\end{proof}

We next exclude the two smallest possible shapes of the additive
order.

\begin{proposition}\label{prop:additive-chain}
If the additive order of a finite ai-semiring $S$ is a chain, then
$S$ is not INFB.
\end{proposition}
\begin{proof}
We first prove that every finite monoid $M$ equipped with a total
order preserved by multiplication on both sides is not INFB.  Write
$1$ for its identity.  For each $a\in M$, the sequence
$a,a^2,a^3,\ldots$ is decreasing if $a\leq1$ and increasing if
$1\leq a$.  Hence there is a common positive integer $N$ such that
\begin{equation}\label{eq:order-stabilization}
  a^N=a^{N+1}\qquad(a\in M).
\end{equation}
Every $a^N$ is idempotent, and every subgroup of $M$ is trivial.

We claim that, for all idempotents $p,q\in M$,
\begin{equation}\label{eq:order-conservative}
  pq,qp\in\{p,q\}.
\end{equation}
Assume that $p\leq q$.  If $q\leq1$, then
$p=p^2\leq pq\leq p1=p$, and similarly $qp=p$.  If $1\leq p$,
the same calculation gives $pq=qp=q$.  In the remaining case
$p\leq1\leq q$, put $r=pq$.  Then
$p\leq r\leq q$, $pr=r$, and $rq=r$.  If $r\leq1$, then
$r=pr\leq p$; if $1\leq r$, then $r=rq\geq q$.  Thus $r\in\{p,q\}$.
The same argument in the opposite monoid treats $qp$ and proves
\eqref{eq:order-conservative}.

We also need the following elementary observation.  If idempotents
$p,q$ satisfy $pq=qp=q$ and $p=uqv$, then $p=q$.  Indeed, in the
finite monoid $pMp$, whose identity is $p$,
\[
  p=(pup)q(pvp).
\]
Every factor of a product equal to the identity of a finite monoid is
a unit.  Thus $q$ is an idempotent unit of $pMp$, whence $q=p$.

For arbitrary $x,y\in M$, put $p=(xy)^N$ and $q=(yx)^N$.  Then
$p,q$ are idempotent and $p=xqy$.  By
\eqref{eq:order-conservative}, either $pq=p$ or $qp=p$, in which case
$pqp=p$, or else $pq=qp=q$.  The preceding observation then yields
$p=q$, so again $pqp=p$.  Consequently,
\begin{equation}\label{eq:order-DA-identities}
  x^{N+1}\approx x^N,\qquad
  (xy)^N(yx)^N(xy)^N\approx(xy)^N
\end{equation}
hold in $M$.

These identities place $M$ in the standard class $\mathbf{DA}$.  In
the structural characterization of this class
\cite[p.~89]{KrebsStraubing2012}, if $f$ is idempotent and $P_f$ is
the submonoid generated by all $a$ such that $f\in MaM$, then
\begin{equation}\label{eq:order-DA-local}
  fP_ff=\{f\}.
\end{equation}
If $a$ divides $f$, both $a$ and $a^2$ belong to $P_f$, so
$faf=fa^2f=f$.  All subgroups of $M$ are trivial; hence
\eqref{eq:order-sapir-test} holds with $d=1$.  By
Lemma~\ref{lem:sapir-order-test}(2), $M$ is not INFB.

Now put $T=S_\cdot$.  For every idempotent $e\in T$, the local monoid
$eTe$ inherits a total order compatible with multiplication, and so
is not INFB by the preceding argument.  Lemma~\ref{lem:local-monoids}
shows that $T$ is not INFB.  The conclusion follows from
Lemma~\ref{lem:infb-reduct}.
\end{proof}

Recall that the height of a finite partially ordered set is the
greatest number of elements in a chain.

\begin{proposition}\label{prop:additive-height-two}
If the additive order of a finite ai-semiring $S$ has height at most
two, then $S$ is not INFB.
\end{proposition}
\begin{proof}
Put $T=S_\cdot$, fix an idempotent $e\in T$, and consider the local
monoid $M=eTe$.  It is also a subsemiring of $S$.  Let
$c=\sum_{x\in M}x$, the greatest element of $M$, and let $d$ be the
least common multiple of the exponents of the subgroups of $M$.  We
verify \eqref{eq:order-sapir-test}.

Suppose first that $e=c$.  If $f$ is idempotent and $f=uav$, then
$f\leq a\leq e$, so
\[
  f=f^{k+2}\leq fa^kf\leq fef=f\qquad(k\geq1).
\]
Thus $fa^kf=f$ for every $k\geq1$.

Now suppose that $e<c$.  Since $c=ce\leq c^2\leq c$, we have
$c^2=c$.  If $f\notin\{e,c\}$ is idempotent, then $e+f=c$ and
\[
  cf=(e+f)f=f,\qquad fc=f(e+f)=f.
\]
For every $x\in M$, it follows that $xf,fx\leq f$.  The element $f$
is additively minimal, since the order has height at most two, and
hence $xf=fx=f$.  Thus $f$ is a multiplicative zero.  There is at
most one such additional idempotent; denote it by $z$ when it exists.

We claim that
\begin{equation}\label{eq:order-top-sandwich}
  ca^kc=c\qquad(k\geq1)
\end{equation}
for every $a\in M\setminus\{z\}$.  The claim is clear for $a=c$.
If $a\ne c$, then $a\leq ca,ac\leq c$; the height assumption gives
$ca,ac\in\{a,c\}$.  If both equal $a$, then, for every $x\in M$,
$xa,ax\leq a$, and minimality makes $a$ a multiplicative zero.
Hence $a=z$.  Otherwise $ca=c$ or $ac=c$, and
\eqref{eq:order-top-sandwich} follows from $c^2=c$.

It remains to consider the possible idempotents $f$.  If $f=e$ and
$a$ divides $e$, then $a$ is a unit of the finite monoid $M$, so
$a^{d+1}=a$ and
\[
  eae=ea^{d+1}e=a\in H_e.
\]
If $f=c$ and $a$ divides $c$, then $a\ne z$, since $uzv=z\ne c$;
therefore \eqref{eq:order-top-sandwich} gives
$cac=ca^{d+1}c=c\in H_c$.  Finally, if $f=z$, then
$zaz=za^{d+1}z=z\in H_z$.  Thus
\eqref{eq:order-sapir-test} holds in every case, and
Lemma~\ref{lem:sapir-order-test}(2) shows that $M=eTe$ is not INFB.
As $e$ was arbitrary, Lemmas~\ref{lem:local-monoids} and
\ref{lem:infb-reduct} show that $S$ is not INFB.
\end{proof}

Combining these results gives a convenient summary.

\begin{corollary}\label{cor:infb-order-restrictions}
Let $S$ be a finite INFB ai-semiring, and let
$\top=\sum_{s\in S}s$.  Then its additive order has width at least
two and height at least three.  If $S$ has a multiplicative identity,
that identity is neither the greatest nor the least element.  Every
multiplicative subgroup of $S$ is an antichain.
\end{corollary}
\begin{proof}
Propositions~\ref{prop:additive-chain} and
\ref{prop:additive-height-two} give the assertions about width and
height.  The remaining statements follow from
Propositions~\ref{prop:extreme-identity} and
\ref{prop:subgroups-antichain}.
\end{proof}

\begin{remark}
Propositions~\ref{prop:additive-chain} and
\ref{prop:additive-height-two} prove that the semirings in question
are not INFB; they do not assert that those semirings are themselves
finitely based.
\end{remark}

\paragraph{Two six-element examples.}\label{sec:examples}

We now apply Theorem~\ref{thm:main} to two familiar six-element
ai-semirings.  The first one is the Brandt ai-semiring
\[
\B=\{0,1,c,d,cd,dc\}.
\]
Its multiplication is that of the Brandt monoid
\[
B_2^1=\langle c,d,1,0\mid cdc=c,\ dcd=d,\ c^2=d^2=0\rangle,
\]
and its addition is the join operation of the upper semilattice in
Figure~\ref{fig:B-structure}.  The element $0$ is the greatest element
of the additive order and is also the multiplicative zero.

\begin{figure}[ht]
\centering
\begin{minipage}[c]{0.38\textwidth}
\centering
\begin{tikzpicture}[
  x=0.78cm,y=0.88cm,
  element/.style={font=\normalsize,inner sep=2pt,fill=white}
]
\coordinate (top) at (0,2.5);
\coordinate (bot) at (0,0);
\coordinate (c0)  at (-2.7,1.3);
\coordinate (cd0) at (-0.9,1.3);
\coordinate (dc0) at (0.9,1.3);
\coordinate (d0)  at (2.7,1.3);
\draw (bot)--(cd0) (bot)--(dc0);
\draw (c0)--(top) (cd0)--(top) (dc0)--(top) (d0)--(top);
\node[element] at (top) {$0$};
\node[element] at (c0) {$c$};
\node[element] at (cd0) {$cd$};
\node[element] at (dc0) {$dc$};
\node[element] at (d0) {$d$};
\node[element] at (bot) {$1$};
\end{tikzpicture}
\par\smallskip
\textup{(a) Additive order}
\end{minipage}
\hfill
\begin{minipage}[c]{0.58\textwidth}
\centering
\(
\begin{array}{c|cccccc}
\cdot&0&1&c&d&cd&dc\\ \hline
0&0&0&0&0&0&0\\
1&0&1&c&d&cd&dc\\
c&0&c&0&cd&0&c\\
d&0&d&dc&0&d&0\\
cd&0&cd&c&0&cd&0\\
dc&0&dc&0&d&0&dc
\end{array}
\)
\par\smallskip
\textup{(b) Multiplication table}
\end{minipage}
\caption{The additive order and multiplication table of $\B$.}
\label{fig:B-structure}
\end{figure}

The second semiring is
\[
\A=\{0,1,a,b,ab,ba\}.
\]
Its multiplicative reduct is the monoid with zero
\begin{equation}\label{eq:A-presentation}
A_2^1=
\langle a,b,1,0\mid
a^2=a,\ b^2=0,\ aba=a,\ bab=b\rangle.
\end{equation}
Addition in $\A$ is the join operation of the upper semilattice shown
in Figure~\ref{fig:A-structure}.  Thus
\[
a<1<ab,ba<b<0,
\]
where $ab$ and $ba$ are incomparable.  This order is compatible with
multiplication; equivalently, multiplication distributes over its join.

\begin{figure}[ht]
\centering
\begin{minipage}[c]{0.38\textwidth}
\centering
\begin{tikzpicture}[
  x=1.0cm,y=0.7cm,
  element/.style={font=\normalsize,inner sep=2pt,fill=white}
]
\coordinate (z)  at (0,5);
\coordinate (b)  at (0,4);
\coordinate (ab) at (-1.25,3);
\coordinate (ba) at (1.25,3);
\coordinate (one) at (0,2);
\coordinate (a) at (0,1);
\draw (a)--(one)--(ab)--(b)--(z);
\draw (one)--(ba)--(b);
\node[element] at (z) {$0$};
\node[element] at (b) {$b$};
\node[element] at (ab) {$ab$};
\node[element] at (ba) {$ba$};
\node[element] at (one) {$1$};
\node[element] at (a) {$a$};
\end{tikzpicture}
\par\smallskip
\textup{(a) Additive order}
\end{minipage}
\hfill
\begin{minipage}[c]{0.58\textwidth}
\centering
\(
\begin{array}{c|cccccc}
\cdot&0&1&a&b&ab&ba\\ \hline
0&0&0&0&0&0&0\\
1&0&1&a&b&ab&ba\\
a&0&a&a&ab&ab&a\\
b&0&b&ba&0&b&0\\
ab&0&ab&a&0&ab&0\\
ba&0&ba&ba&b&b&ba
\end{array}
\)
\par\smallskip
\textup{(b) Multiplication table}
\end{minipage}
\caption{The additive order and multiplication table of $\A$.}
\label{fig:A-structure}
\end{figure}

\begin{corollary}\label{cor:BA-infb}
The ai-semirings $\B$ and $\A$ are inherently nonfinitely based.
\end{corollary}

\begin{proof}
Every Zimin word is minimal for $\B$ by~\cite[Lemma~3.1]{RYZ}.
Since $\B$ is finite, Theorem~\ref{thm:main} shows that $\B$ is
inherently nonfinitely based.

It is known that $\B$ is a homomorphic image of a subsemiring of
$\A\times\A$; see~\cite[Proposition~3]{GusevVolkovChains}.  Hence
$\B\in\V(\A)$.  Suppose that $\A\models \bu\preceq Z_m$.  Every identity
of $\A$ holds in every member of $\V(\A)$, and therefore
$\B\models \bu\preceq Z_m$.  The minimality of $Z_m$ for $\B$ gives the
literal equality $\bu=Z_m$.  Thus every Zimin word is minimal for $\A$.
The semiring $\A$ is finite, so another application of
Theorem~\ref{thm:main} proves that $\A$ is inherently nonfinitely based.
\end{proof}

\section[The ai-semiring A-bar is not INFB]
{The ai-semiring \texorpdfstring{$\Abar$}{A-bar} is not
inherently nonfinitely based}
\label{sec:not-infb}

Let $X$ be a countably infinite alphabet.  We write $P_f(X^+)$ for
the free ai-semiring of finite nonempty subsets of $X^+$, with union
as addition and setwise concatenation as multiplication, and identify
a word with its singleton subset.

Let $\Abar$ denote the six-element ai-semiring
$\{0,1,a,b,ab,ba\}$.  Its additive order and multiplicative Cayley
table are given in Figure~\ref{fig:Abar-structure} and
Table~\ref{tab:Abar-multiplication}, respectively.  The notation
$\Abar$ distinguishes this additive structure from the
ai-semiring $\A$ considered in Section~\ref{sec:examples}; the two
algebras have the same multiplicative reduct.

\begin{figure}[ht]
\centering
\begin{tikzpicture}[
  x=1.0cm,y=0.7cm,
  element/.style={font=\normalsize,inner sep=2pt,fill=white}
]
\coordinate (z)  at (0,5);
\coordinate (b)  at (0,4);
\coordinate (ab) at (-1.25,3);
\coordinate (ba) at (1.25,3);
\coordinate (one) at (0,2);
\coordinate (a) at (0,1);
\draw (a)--(one)--(ab)--(b)--(z);
\draw (one)--(ba)--(b);
\node[element] at (z) {$a$};
\node[element] at (b) {$1$};
\node[element] at (ab) {$ab$};
\node[element] at (ba) {$ba$};
\node[element] at (one) {$b$};
\node[element] at (a) {$0$};
\end{tikzpicture}
\caption{The additive order of $\Abar$.}
\label{fig:Abar-structure}
\end{figure}

\begin{table}[ht]
\centering
\caption{The multiplication table of $\Abar$.}
\label{tab:Abar-multiplication}
\begin{tabular}{c|cccccc}
$\cdot$&$0$&$1$&$a$&$b$&$ab$&$ba$\\ \hline
$0$&$0$&$0$&$0$&$0$&$0$&$0$\\
$1$&$0$&$1$&$a$&$b$&$ab$&$ba$\\
$a$&$0$&$a$&$a$&$ab$&$ab$&$a$\\
$b$&$0$&$b$&$ba$&$0$&$b$&$0$\\
$ab$&$0$&$ab$&$a$&$0$&$ab$&$0$\\
$ba$&$0$&$ba$&$ba$&$b$&$b$&$ba$
\end{tabular}
\end{table}

Dolinka~\cite{dol3} proposed the following problem for this semiring.

\begin{problem}\label{pro26092866}
Is the ai-semiring $\Abar$ inherently nonfinitely based?
\end{problem}

In this section, we give a negative answer to Problem~\ref{pro26092866}
by showing that $\Abar$ is not inherently nonfinitely based.  We shall
use Lemma~\ref{lem:infb-reduct}, which relates an ai-semiring to its
multiplicative reduct.

It is well-known that every semigroup of order less than six is finitely based.
By Lemma~\ref{lem:infb-reduct}, we immediately deduce the following.

\begin{corollary}\label{coro26092610}
No ai-semiring of order less than six is inherently nonfinitely based.
\end{corollary}

It is natural to ask whether the converse of Lemma~\ref{lem:infb-reduct} holds.

\begin{problem}\label{problem26092601}
Let $S$ be an ai-semiring.
If the multiplicative reduct of $S$ is inherently nonfinitely based,
is $S$ itself inherently nonfinitely based?
\end{problem}

We know that the multiplicative semigroup $(\Abar, \cdot)$ is inherently nonfinitely based.
Theorem~\ref{thm26092715} below solves Problem~\ref{problem26092601} in the negative
by showing that the ai-semiring $\Abar$ is not inherently nonfinitely based.
To the best of our knowledge, this is the first such example;
moreover, by Corollary~\ref{coro26092610},
it is the smallest possible one.

\begin{remark}\label{remark26092801}
For a word $\bw$ of length at least two, let $h(\bw)$ and $t(\bw)$ denote its first and last letters, respectively.  Let $\bw$ be such a word,
and let $\varphi \colon P_f(X^+) \to \Abar$ be a homomorphism.
If $\varphi(\bw)\neq 0$, $\varphi(h(\bw))\neq 1$, and $\varphi(t(\bw))\neq 1$,
then $\varphi(\bw)$ is determined solely by
$\varphi(h(\bw))$ and $\varphi(t(\bw))$.

More precisely, $\varphi(\bw)=\varphi(h(\bw)) \star \varphi(t(\bw))$,
where the operation $\star$ is given by Table~\ref{tbstar}.
In particular, if $\varphi(h(\bw)), \varphi(t(\bw)) \in \{a, ab, ba\}$,
then $\varphi(h(\bw)) \star \varphi(t(\bw))=\varphi(h(\bw)) \cdot \varphi(t(\bw))$.

\begin{table}[ht]
\caption{The Cayley table of the operation $\star$.} \label{tbstar}
\begin{tabular}{c|cccc}
$\star$ &$a$&$b$&$ab$&$ba$\\
\hline
$a$  &$a$ &$ab$&$ab$&$a$ \\
$b$  &$ba$&$b$ &$b$ &$ba$\\
$ab$ &$a$ &$ab$&$ab$&$a$ \\
$ba$ &$ba$&$b$ &$b$ &$ba$\\
\end{tabular}
\end{table}
\end{remark}

\begin{lemma}\label{lemma26092850}
Let $\bv$ and $\bv'$ be words such that
\[
\bv=\bp_1 x \bp_2 x \bp_3 x \bp_4
\quad\text{and}\quad
\bv'=\bp_1 x \bp_2 \bp_3 x \bp_4,
\]
where $\bp_1$, $\bp_2$, $\bp_3$, and $\bp_4$ may be empty words.
Further let $\varphi \colon P_f(X^+) \to \Abar$ be a homomorphism.
If $\varphi(\bv) \neq 0$, then $\varphi(\bv')=0$ or $\varphi(\bv')=\varphi(\bv)$.
\end{lemma}
\begin{proof}
If $\varphi(\bv')=0$, then the result holds.
So assume that $\varphi(\bv') \neq 0$.
If $\varphi(x)=1$, then $\varphi(\bv)=\varphi(\bv')$,
since $1$ is the multiplicative identity of $\Abar$.
Now suppose that $\varphi(x)\neq 1$.
Let $\bv_1$ and $\bv'_1$ be the words obtained from $\bv$ and $\bv'$,
respectively, by deleting all variables whose values under $\varphi$ are $1$.
Then $\varphi(\bv_1)=\varphi(\bv)$ and $\varphi(\bv'_1)=\varphi(\bv')$.
Note that $\bv_1$ and $\bv'_1$ have the forms
\[
\bp'_1 x \bp'_2 x \bp'_3 x \bp'_4
\quad\text{and}\quad
\bp'_1 x \bp'_2 \bp'_3 x \bp'_4,
\]
respectively.
By Remark~\ref{remark26092801}, $\varphi(\bv_1)=\varphi(\bv'_1)$.
Hence $\varphi(\bv)=\varphi(\bv')$, as required.
\end{proof}

Let $\sigma$ denote the identity
\[
Z_5 \approx \ba_4x_5Z_4+\bb_4x_5Z_4,
\]
where $\ba_4=x_1x_2x_3Z_2x_4Z_3$ and $\bb_4=Z_2x_3x_1x_1x_4Z_3$.
Let $\mathcal{V}_\sigma$ be the ai-semiring variety defined by the identity $\sigma$.
Then $\mathcal{V}_\sigma$ is finitely based.
Note that the word $\ba_4$ is obtained
from the word $Z_4$ by deleting the second occurrence of $x_1$ in $Z_4$,
and $\bb_4$ is obtained from $Z_4$ by deleting the second occurrence of $x_2$ in its initial copy of $Z_3$.

\begin{proposition}\label{pro26092820}
The variety $\mathcal{V}_\sigma$ contains $\Abar$.
\end{proposition}

\begin{proof}
It suffices to show that $\Abar$ satisfies the identity $\sigma$.
Let $\varphi \colon P_f(X^+) \to \Abar$ be an arbitrary homomorphism.
If $\varphi(Z_4)=0$, then $\varphi(Z_5)= 0$ and $\varphi(\ba_4x_5Z_4)=\varphi(\bb_4x_5Z_4)=0$,
so
\[
\varphi(Z_5)
= \varphi(\ba_4x_5Z_4+\bb_4x_5Z_4).
\]

Now suppose that $\varphi(Z_4)\neq 0$.
By Lemma~\ref{lemma26092850},
$\varphi(\ba_4), \varphi(\bb_4) \in \{0, \varphi(Z_4)\}$,
and hence
$\varphi(\ba_4x_5Z_4), \varphi(\bb_4x_5Z_4) \in \{0, \varphi(Z_5)\}$.
Thus it suffices to prove that at least one of
$\varphi(\ba_4x_5Z_4)$ and $\varphi(\bb_4x_5Z_4)$ equals $\varphi(Z_5)$.
For this, it is enough to show that at least one of
$\varphi(\ba_4)$ and $\varphi(\bb_4)$ equals $\varphi(Z_4)$.

If $\varphi(x_1)=1$, then $\varphi(\ba_4)=\varphi(Z_4)$, since $1$ is the multiplicative identity.
Now suppose that $\varphi(x_1) \in \{a, ab, ba\}$. Then $\varphi(x_1x_1)=\varphi(x_1)$.
Since $\varphi(Z_4)\neq 0$, it follows that $\varphi(Z_2)\neq 0$, and so $\varphi(x_1x_2x_1)\neq 0$.
By Remark~\ref{remark26092801}, $\varphi(x_1x_2x_1)=\varphi(x_1x_1)$.
Thus
\[
\varphi(\bb_4)=\varphi(Z_2x_3x_1x_1x_4Z_3)
=\varphi(Z_2x_3x_1x_2x_1x_4Z_3)=\varphi(Z_4).
\]

Finally, suppose that $\varphi(x_1)=b$.
From the multiplicative Cayley table of $\Abar$,
one obtains that $bsb=b$ if and only if $s=a$.
Since the word $Z_4$ contains the subwords $x_1x_2x_1$, $x_1x_3x_1$, and $x_1x_4x_1$,
it follows that $\varphi(x_2)=\varphi(x_3)=\varphi(x_4)=a$, and so
\[
\varphi(x_2x_3)=aa=a=aba=\varphi(x_2x_1x_3).
\]
Hence
\[
\varphi(\ba_4)=\varphi(x_1x_2x_3Z_2x_4Z_3)=\varphi(x_1x_2x_1x_3Z_2x_4Z_3)=\varphi(Z_4).
\]

We have shown that at least one of $\varphi(\ba_4)$ and $\varphi(\bb_4)$
equals $\varphi(Z_4)$.
This completes the proof.
\end{proof}
\begin{proposition}\label{pro26092821}
The variety $\mathcal{V}_\sigma$ is locally finite.
\end{proposition}

\begin{proof}
Let $R=\langle X\rangle$ be an ai-semiring in $\mathcal V_\sigma$
generated by a finite set $X$.  Since $Z_5$ is unavoidable, there is
an integer $k$ such that every word over $X$ of length greater than
$k$ contains a value of $Z_5$ as a factor.

We claim that the value in $R$ of every word over $X$ is a finite sum
of values of words of length at most $k$.  It is enough to prove the
claim for a word $\bw$ of length greater than $k$.  Write
\[
  \bw=\bp\psi(Z_5)\bq,
\]
where $\bp,\bq\in X^*$ and $\psi:X^+\to X^+$ is a substitution.  The
identity $\sigma$ gives
\begin{align*}
\bw
&\approx
\bp\psi(\ba_4x_5Z_4)\bq
+\bp\psi(\bb_4x_5Z_4)\bq.
\end{align*}
Both words on the right are strictly shorter than $\bw$.  Repeating
this reduction and arguing by induction on length proves the claim.

By distributivity, every element of $R$ is a finite sum of values of
words over $X$.  The claim allows all of these words to be chosen with
length at most $k$.  There are only finitely many such words, and
addition is commutative and idempotent.  Hence only finitely many sums
of their values are possible.  Thus $R$ is finite, and
$\mathcal V_\sigma$ is locally finite.
\end{proof}

We now arrive at the main result of this section.

\begin{theorem}\label{thm26092715}
The six-element ai-semiring $\Abar$ is not inherently nonfinitely based.
\end{theorem}
\begin{proof}
This follows immediately from Propositions \ref{pro26092820} and \ref{pro26092821}.
\end{proof}
\section[A finite identity basis for A-bar]
{A finite identity basis for
\texorpdfstring{$\Abar$}{A-bar}}
\label{sec:finite-basis}

We now determine the finite basis status of $\Abar$.  As usual, it
suffices to consider identities of the form
\[
  \bu\approx\bu+\bq,
  \qquad
  \bu=\bu_1+\cdots+\bu_n,
\]
where $\bu_1,\ldots,\bu_n,\bq\in X^+$.  We shall derive such an
identity by first adjoining $\bq$ together with its square and then
absorbing the square.

For a word or term $\bw$, write $c(\bw)$ for its content. If $\bu$ is
a polynomial and $\bq$ is a word, let
\[
  D_{\bq}(\bu)=
  \{\bu_i:1\leq i\leq n,\ c(\bu_i)\subseteq c(\bq)\}.
\]
When this set is nonempty, we also use $D_{\bq}(\bu)$ for the sum of
its members. We write $\bs\leq\bt$ for the identity
$\bs+\bt\approx\bt$. An absorption
$\bq\leq\bu$ is called \emph{retained} if
$c(\bu)\subseteq c(\bq)$.

\begin{lemma}\label{lem:equational-reduction}
Let $\bu\approx\bu+\bq$ be an identity of $\Abar$.  Then
$D_{\bq}(\bu)$ is nonempty,
\[
  \bigcup_{\bv\in D_{\bq}(\bu)}c(\bv)=c(\bq),
\]
and $\Abar$ satisfies
\[
  D_{\bq}(\bu)\approx D_{\bq}(\bu)+\bq.
\]
Consequently, to obtain an identity basis for $\Abar$, it is enough
to derive every valid absorption $\bq\leq\bu$ with
$c(\bu)\subseteq c(\bq)$.
\end{lemma}
\begin{proof}
Put $C=c(\bq)$.  Assign $1$ to the variables in $C$ and $0$ to all
other variables.  Then $\bq$ takes the value $1$, and every summand
outside $D_{\bq}(\bu)$ takes the value $0$.  Thus $D_{\bq}(\bu)$ is
nonempty.

Suppose that a variable $x\in C$ occurs in none of the retained
summands.  In the preceding assignment, change the value of $x$ to
$a$.  The retained sum still takes the value $1$, whereas $\bq$
takes the value $a$.  Since $a\nleq1$, this contradicts the assumed
identity.  This proves the assertion about contents.

Now take any assignment to the variables in $C$ and extend it by
assigning $0$ outside $C$.  Every discarded summand vanishes, so the
assumed identity gives
$D_{\bq}(\bu)\approx D_{\bq}(\bu)+\bq$ under that assignment.
Hence this is an identity of $\Abar$.

Finally, distributivity expresses every semiring term as a
polynomial.  An identity between two polynomials follows from the
absorption of each summand by the opposite polynomial.  Once a
retained absorption has been derived, the discarded summands can
be restored by adding them to both sides.  This proves the last
assertion.
\end{proof}

Let AI denote the six ai-semiring identities
\begin{equation}\label{eq:h:002}
x+y=y+x,\quad (x+y)+z=x+(y+z),\quad x+x=x,
\end{equation}
\begin{equation}\label{eq:h:003}
(xy)z=x(yz),\quad x(y+z)=xy+xz,\quad (x+y)z=xz+yz.
\end{equation}
We use the following five further identities, denoted by Base:
\begin{equation}\label{eq:h:004}
x^3=x^2,\quad x^2\leq x,\quad xy=x^2y+xy^2,\quad
(x+y)^2=x^2+y^2,\quad xy\leq x+y.
\end{equation}
Their validity is proved in
Lemma~\ref{lem:elementary-structure} below.

The following consequence of square decrease will be used repeatedly.
It allows letters already present in a word to be absorbed inside
a product without adding any new identities to the basis.

\begin{lemma}\label{band:content-contraction}\label{lem:content-sandwich}
AI and square decrease \(x^2\le x\) derive
\begin{equation}\label{eq:b:036}
\bu\bv\bu\le\bu
\end{equation}
whenever \(\bu\) is a nonempty word and \(\bv\) is a term on \(c(\bu)\).
\end{lemma}
\begin{proof}
First \(xy\le(x+y)^2\le x+y\), so every word is at most the sum of
its letters. It suffices to prove \eqref{eq:b:036} when \(\bv=y\) is a letter
occurring in \(\bu\): the word bound and distributivity then handle every word,
and polynomial expansion handles every term.

Select an occurrence of \(y\). If \(\bu=\bp y\bs\) with both surrounding
words \(\bp,\bs\) nonempty, then
\begin{equation}\label{eq:b:037}
\begin{aligned}
\bu y\bu&=\bp y(\bs y\bp)y\bs\\
   &\le\bp y(\bs+y+\bp)y\bs\\
   &=\bp(y\bs)^2+\bp y^3\bs+(\bp y)^2\bs\\
   &\le\bp y\bs=\bu.
\end{aligned}
\end{equation}
The final inequalities are square decrease applied to \(y\bs,y,\bp y\),
with iteration for \(y^3\le y\). If \(\bu=y\bs\) at the left endpoint,
with \(\bs\) nonempty, then
\(\bu y\bu=y\bs y^2\bs\le(y\bs)^2\le y\bs\).
If \(\bu=\bp y\) at the right endpoint, then
\(\bu y\bu=\bp y^2\bp y\le(\bp y)^2\le\bp y\).
If \(\bu=y\), use \(y^3\le y\).
This completes the proof.
\end{proof}

Let $\mathcal E$ consist of AI and Base, together with the following
finite families: the 48 lifting identities, the 17 coordinate
identities, the eleven scalar identities, the eight guarded band
identities, the ten transfer identities, and the four two-path
identities.  Their definitions and derivations are given in the
appendix.  Each abbreviation there has a fixed expansion into
$\{+,\cdot\}$, and every substitution uses a nonempty term.
Thus $\mathcal E$ contains at most
\[
  6+5+48+17+11+8+10+4=109
\]
identities, each involving at most seven variables.

\begin{theorem}\label{main:Abar-fb}
The six-element ai-semiring $\Abar$ is finitely based.
More precisely, $\mathcal E$ is a finite identity basis for $\Abar$.
\end{theorem}
\begin{proof}
The validity results in the appendix show that every identity in
$\mathcal E$ holds in $\Abar$.  It remains to prove that every
identity of $\Abar$ follows from $\mathcal E$.

Let $\bu\approx\bu+\bq$ be such an identity.  By
Lemma~\ref{lem:equational-reduction}, we may assume that
$c(\bu)\subseteq c(\bq)$.

First suppose that $c(\bq)=\{x\}$.  If $\bq=x$, the substitution
$x=b$ shows that one of the summands of $\bu$ must be $x$.
The desired absorption then follows from additive idempotence.
If $\bq=x^m$ with $m\geq2$, choose any summand $x^r$ of $\bu$.
The power identity and square decrease give $x^m=x^2\leq x^r$,
so again $\bu\approx\bu+\bq$ is derivable.

Now suppose that $|c(\bq)|\geq2$, and put
\[
  H(\bq)=\sum_{x\in c(\bq)}\bq x\bq.
\]
The absorption $\bq\leq\bu$ holds in the five-element subsemiring
$S=\{0,b,ab,ba,1\}$.  By Lemma~\ref{lem:finite-lifting},
\[
  \mathcal E\vdash\bq\leq\bu+H(\bq).
\]
For each $x\in c(\bq)$, Lemma~\ref{lem:coordinate-terms} gives
\[
  \mathcal E\vdash\bq x\bq\leq\bu+\bq^2.
\]
Adding these inequalities yields
$\mathcal E\vdash\bq\leq\bu+\bq^2$, or equivalently
\[
  \mathcal E\vdash
  \bu+\bq^2\approx\bu+\bq^2+\bq.
\]

Since square decrease holds in $\Abar$, the absorption
$\bq^2\leq\bu$ is valid in $\Abar$.
Lemma~\ref{completeness:square} therefore gives
\[
  \mathcal E\vdash\bu\approx\bu+\bq^2.
\]
Consequently,
\begin{align*}
  \bu
  &\approx\bu+\bq^2\\
  &\approx\bu+\bq^2+\bq\\
  &\approx\bu+\bq.
\end{align*}
Thus every retained valid absorption is derivable from
$\mathcal E$.  Lemma~\ref{lem:equational-reduction} now shows that
$\mathcal E$ is an identity basis for $\Abar$.
\end{proof}

\appendix
\section{Derivations for the finite identity basis}
\label{app:finite-basis}
We give the finite families used in Section~\ref{sec:finite-basis}
and prove the reductions on which its absorption argument depends.
We use the completeness theorem for equational logic: an identity
holds in every model of an identity set if and only if it is derivable
from that set; see~\cite{BurrisSankappanavar1981}. Auxiliary words may
depend on the target identity, but the identity set $\mathcal E$ is fixed.

\begin{lemma}[Elementary structure and base identities]\label{lem:elementary-structure}
The table defines an ai-semiring. Every element except \(b\) is
multiplicatively idempotent, \(b^2=0\), and Base is valid. A product equal
to \(1\) has every factor equal to \(1\). For every nonzero step map \(d\)
and every element \(v\), \(dvd\in\{0,d\}\), and \(dad=d\).
\end{lemma}
\begin{proof}
Use the monotone maps of \(\{0<1<2\}\) fixing zero, with the convention
\((fg)(i)=g(f(i))\). Their pairs \((f(1),f(2))\) are
\begin{equation}\label{eq:h:005}
0=(0,0),\quad b=(0,1),\quad ab=(1,1),\quad
ba=(0,2),\quad 1=(1,2),\quad a=(2,2).
\end{equation}
Pointwise maximum gives exactly the specified order. Substitution of these
six pairs in \((fg)(i)=g(f(i))\) gives the table. Composition is associative.
Every monotone map of a chain preserves binary maximum, so composition
distributes over pointwise maximum on both sides. This proves AI.

The four maps \(a,b,ab,ba\) are step maps \(d_{i,j}\), taking value zero
below threshold \(i\in\{1,2\}\) and value \(j\in\{1,2\}\) at and above it.
Specifically \(a=d_{1,2},b=d_{2,1},ab=d_{1,1},ba=d_{2,2}\). Their product is
\(d_{i,j}d_{k,l}=0\) if \(j<k\), and \(d_{i,l}\) otherwise. This formula
proves the sandwich assertions, including when \(v=0\) or \(1\). Every
nonidentity map has image of size at most two. Composition cannot increase
image size, whereas \(1\) has image of size three; hence a product equal to
\(1\) has only identity factors.

The table gives the power law and square decrease. Both summands of
\(x^2y+xy^2\) are at most \(xy\). If either factor is idempotent, one of
these summands equals \(xy\). Otherwise both are \(b\), and every term is
zero. Thus the third base law holds. Square decrease at \(x+y\), followed
by distributivity, gives \(xy\leq x+y\).

To prove square additivity it suffices to show \(xy\leq x^2+y^2\) and use
the same inequality with the variables interchanged. If both factors are
idempotent, use the preceding product bound. If both are \(b\), the product
is zero. If one factor is \(b\) and the other is an idempotent \(z\), both
\(bz\) and \(zb\) lie below \(z\): this follows for
\(z=0,1,a,ab,ba\) from the displayed table. These cases exhaust all pairs.
Expanding \((x+y)^2\) proves the fourth law.
\end{proof}
\paragraph{Definitions of the finite identity families.}\label{sec:Abar-fb-setup}
We use the abbreviations AI and Base introduced in
Section~\ref{sec:finite-basis}.  The auxiliary terms below are all
terms in the signature \(\{+,\cdot\}\); they do not introduce new
operations or constants.

\paragraph{Conventions for the finite schemes.}
An unparenthesized product or sum is left-associated. In each finite
identity scheme, an optional context is either omitted or replaced by
a fresh variable. All scalar meets are taken in the displayed
left-associated order.  These conventions turn every scheme into a
finite list of ordinary identities in the original signature.

\paragraph{Cuts and the five-element subsemiring.}\label{sec:3}

\begin{lemma}[A uniform cut expansion]\label{lem:cut-expansion}
AI and the first three base identities derive, for
\(\bw=x_1\cdots x_n\),
\begin{equation}\label{eq:h:006}
\bw=\sum_{i=1}^{n}\bp_i^2x_i\bs_i^2,
\tag{\text{H-Cut}}
\end{equation}
where \(\bp_i=x_1\cdots x_{i-1}\), \(\bs_i=x_{i+1}\cdots x_n\), and absent
contexts are omitted. Every cut summand is separately at most \(\bw\).
\end{lemma}
\begin{proof}
Applying \(uv=u^2v+uv^2\) to \(u=x^2y^2,v=y^2z^2\), and reducing every
positive power at least two to the square, gives
\begin{equation}\label{eq:h:007}
x^2y^2z^2=(x^2y^2)^2z^2+x^2(y^2z^2)^2.
\tag{\text{H-3.1}}
\end{equation}
For nonempty terms \(\bx,\ba,\bt\), first apply that base law to \(\bx,\ba^2\bt\),
and then to \(\bx^2\ba^2,\bt\). This gives
\begin{equation}\label{eq:h:008}
\bx\ba^2\bt=\bx^2\ba^2\bt+\bx(\ba^2\bt)^2,
\qquad \bx^2\ba^2\bt=(\bx^2\ba^2)^2\bt+\bx^2\ba^2\bt^2.
\end{equation}
The first displayed second summand is at most \(\bx(\ba\bt)^2\). Apply \eqref{eq:h:007}
to \(\bx^2\ba^2\bt^2\). Its two summands are at most \((\bx\ba)^2\bt\) and
\(\bx(\ba\bt)^2\), respectively, by square decrease and monotonicity. The term
\((\bx^2\ba^2)^2\bt\) is also at most \((\bx\ba)^2\bt\). Consequently
\begin{equation}\label{eq:h:009}
\bx\ba^2\bt\leq(\bx\ba)^2\bt+\bx(\ba\bt)^2.
\tag{\text{H-3.2}}
\end{equation}

Induct on \(n\). For one letter Cut is identical. Write a longer word as
\(\bw=x\bv\) and expand \(\bv\) inductively. For a cut \(\ba^2y\bb^2\) of \(\bv\), put
\(\bt=y\bb^2\), omitting \(\bb^2\) when absent. If \(\ba\) is present, \eqref{eq:h:009} gives
\begin{equation}\label{eq:h:010}
x\ba^2y\bb^2\leq(x\ba)^2y\bb^2+x(\ba y\bb^2)^2
             \leq(x\ba)^2y\bb^2+x\bv^2.
\end{equation}
If \(\ba\) is absent, the base law gives
\(x\bt=x^2\bt+x\bt^2\leq x^2y\bb^2+x\bv^2\).
Distribution now bounds \(\bw\) by its cut sum. Every cut summand is at most
\(\bw\) by square decrease on its actual nonempty contexts, giving the
reverse inequality. All substitutions above are nonempty terms.
\end{proof}

\paragraph{The 48 lifting patterns.}
Let \(S=\{0,b,ab,ba,1\}\). For a fixed word absorption \(\mathbf L\leq\mathbf R\), its
displayed lift means the single equation
\begin{equation}\label{eq:h:011}
\mathbf L\leq\mathbf R+\sum_{z\in c(\mathbf L)}\mathbf Lz\mathbf L.
\tag{\text{H-Lift-pattern}}
\end{equation}
In this definition the sum is over the literal schematic variables of
the fixed displayed word \(\mathbf L\), including any context variables. The
operation of taking this lift is used only on the following finite lists.

For \(\mathbf P=y,ly,yr,lyr\), we shall also use
\[
  \mathbf P^2y\mathbf P^2\leq\mathbf P^2.
\]
These four inequalities follow from
Lemma~\ref{lem:content-sandwich}, applied to the word \(\mathbf P^2\), and
are not included in the identity basis. We call them the B
inequalities when using them below.

\begin{itemize}
\item I: for each such \(\mathbf P\), and for each independent choice that the
contexts \(A,B\) are absent or a fresh variable, include the lift of
\(A\mathbf P^2B\leq A\mathbf P^2y\mathbf P^2B\).
\item M: for each of the four choices of the outer contexts \(A,B\), include
the lift of
\(AP^2UP^2VP^2B\leq AP^2UVP^2B\). Here \(P,U,V\) are mandatory variables.
\item C: for \(k=1,2,3\), and all independent present/absent choices of
\(A_0,\ldots,A_k\), include the lift of
\(A_0Z_1A_1\cdots Z_kA_k\leq Z_1\cdots Z_k\), with all \(Z_i\) mandatory.

\end{itemize}
Within each pattern fresh variables are distinct unless an identification
is explicitly specified. They may of course receive equal terms under a
later substitution. These three families have \(16,4,28\) members. A pattern
\(\mathbf P\) containing \(y\) represents any actual word containing a selected
occurrence of \(y\), using its nonempty prefix and suffix as the present
contexts.

\begin{lemma}[The semantics needed for lifting]\label{lem:lifting-semantics}
The 48 lifting patterns are valid in \(\Abar\). If \(S\models\bq\leq\mathbf F\), then
for each cut \(\bq=\bp x\bs\) there are a summand \(\mathbf f=\bu y\bv\) of \(\mathbf F\) with
\(c(\mathbf f)\subseteq c(\bq)\), \(c(\bu)\subseteq c(\bp)\), and
\(c(\bv)\subseteq c(\bs)\). Empty prefix/suffix contents are allowed in this
semantic cut condition; the pivot \(y\) need not equal \(x\).
\end{lemma}
\begin{proof}
The multiplication table and the specified joins show that \(S\) is a
subsemiring. In \(S\), define bits
\begin{equation}\label{eq:h:012}
r(f)=[f(2)=2],\quad q_0(f)=[f(2)\geq1],\quad p(f)=[f(1)=1].
\end{equation}
Their tuples \((r,q_0,p)\) for \(0,b,ab,ba,1\) are respectively
\(000,010,011,110,111\). Addition is coordinatewise OR and
\begin{equation}\label{eq:h:013}
r(fg)=r(f)r(g),\quad p(fg)=p(f)p(g),\quad
q_0(fg)=r(f)q_0(g)\vee q_0(f)p(g).
\tag{\text{H-4.1}}
\end{equation}
These follow by evaluating \(g(f(2))\) and \(g(f(1))\). In particular, the
square of a word depends only on its content: its \(r,p\) coordinates are
the conjunctions of the corresponding letter coordinates and its middle
coordinate is their disjunction. If \(y\) occurs in a word \(\mathbf P\), this gives
\(\mathbf P^2y=y\mathbf P^2=\mathbf P^2\) in \(S\). Equivalently, a nonzero square \(ab\) requires
every letter in \(\{ab,1\}\), a square \(ba\) requires every letter in
\(\{ba,1\}\), and a square \(1\) requires every letter equal to \(1\).

The unlifted I patterns are therefore equalities in \(S\). For M put
\(e=P^2\). For \(e=0,1\) its two inner words are respectively zero or
\(UV\). For \(e=ab\), \(ete\) is \(ab\) exactly when \(p(t)=1\), and zero
otherwise; \eqref{eq:h:013} proves \(eUeVe=eUVe\). For \(e=ba\) use \(r\) instead.
Thus M is an equality in \(S\). Every element of \(S\) is at most its
identity \(1\), so deleting the optional factors proves every unlifted C
inequality in \(S\).

For any fixed \(S\)-valid \(\mathbf L\leq\mathbf R\), its lift is \(\Abar\)-valid. If all
variables of \(\mathbf L\) are assigned in \(S\), first set any outside variables
of \(\mathbf R\) to zero and apply \(S\)-validity; restoring their original values
can only increase \(\mathbf R\). Otherwise some \(z\in c(\mathbf L)\) is assigned \(a\).
If \(\mathbf L=0\) the result is immediate. If \(\mathbf L\ne0\), it is a step map rather
than \(1\), so Lemma~\ref{lem:elementary-structure} gives \(\mathbf Lz\mathbf L=\mathbf L\). This proves all the lifted finite
patterns. For B, put \(e=\mathbf P^2\). If \(e=0\) there is nothing to prove; if
\(e\) is a step map its sandwich is at most \(e\); if \(e=1\), all letters
of \(\mathbf P\), including \(y\), equal \(1\). Thus B is also valid.

Induction from \eqref{eq:h:013} gives, for \(\bw=z_1\cdots z_n\),
\begin{equation}\label{eq:h:014}
q_0(\bw)=\bigvee_i
 \left(\bigwedge_{j<i}r(z_j)\right)q_0(z_i)
 \left(\bigwedge_{j>i}p(z_j)\right).
\tag{\text{H-4.2}}
\end{equation}
Fix a cut with prefix/suffix contents \(L_0,R_0\). Assign variables outside
\(c(\bq)\) to zero; to each \(z\in c(\bq)\) assign the unique element of
\(S\) with \(q_0=1,r=[z\in L_0],p=[z\in R_0]\). The cut in \eqref{eq:h:014} makes
\(\bq\) nonzero. A summand \(\mathbf f\) of \(\mathbf F\) is therefore nonzero, uses only
\(c(\bq)\), and has a cut whose preceding letters have \(r=1\) and following
letters have \(p=1\). This is the required witness. Empty Boolean
conjunctions in this evaluation are true; they are not semiring terms.
\end{proof}

\begin{lemma}[Finite lifting from \(S\)]\label{lem:finite-lifting}
For a retained \(S\)-valid absorption \(\bq\leq\mathbf F\), AI, Base and the 48
lifting patterns derive
\begin{equation}\label{eq:h:015}
\bq\leq\mathbf F+H(\bq),\qquad H(\bq)=\sum_{z\in c(\bq)}\bq z\bq.
\tag{\text{H-SL}}
\end{equation}
\end{lemma}
\begin{proof}
The product-to-join bound gives
\(\bu\leq\sum_{z\in c(\bu)}z\) for every nonempty word. Consequently, when
\(c(\bu)\subseteq c(\bp)\), the B patterns give
\begin{equation}\label{eq:h:016}
\bp^2\bu\bp^2\leq\sum_{z\in c(\bu)}\bp^2z\bp^2\leq\bp^2.
\tag{\text{H-Corner}}
\end{equation}
If \(\bt\leq\bq\) has already been derived, and \(\bw\) is a nonempty word on
\(c(\bq)\), then
\begin{equation}\label{eq:h:017}
\bt\bw\bt\leq\bq\bw\bq\leq H(\bq).
\tag{\text{H-Error}}
\end{equation}
Hence an instance of a lifted I, M or C pattern gives
\(\bt\leq\mathbf R'+H(\bq)\) whenever its lower word is separately known to be at
most \(\bq\), and every substituted word uses only \(c(\bq)\). This observation
never multiplies an old error term by a new context.

We first show clean-prefix insertion. Suppose \(\bp,\bv\) are nonempty words
on \(c(\bq)\), \(\bp^2\bv\leq\bq\) is derived, and \(\bu\) is nonempty with
\(c(\bu)\subseteq c(\bp)\). Then
\begin{equation}\label{eq:h:018}
\bp^2\bv\leq\bp^2\bu\bp^2\bv+H(\bq),\qquad
\bp^2\bu\bp^2\bv\leq\bp^2\bv.
\tag{\text{H-Left}}
\end{equation}
The second inequality is Corner. Build \(\bu\) from its last letter by
prepending letters. For one letter \(z\), I gives the first inequality,
with final context \(\bv\), and Error bounds its error. If \(\mathbf U\) has been
built, the current word \(\bp^2\mathbf U\bp^2\bv\) is separately at most \(\bp^2\bv\) by
Corner. Apply I at its first square to insert \(z\), giving
\(\bp^2z\bp^2\mathbf U\bp^2\bv\) up to \(H(\bq)\). This intermediate is also at most \(\bq\),
by B and Corner. Apply M with \(P=\bp,U=z,V=\mathbf U\) and final context \(\bv\),
to obtain \(\bp^2z\mathbf U\bp^2\bv\) up to the same error. Corner bounds this new clean
word independently by \(\bp^2\bv\). Additive transitivity and idempotence
combine the inequalities, proving Left after finitely many letters.

For clean-suffix insertion, with \(\bv,\bs\) nonempty and \(\bv\bs^2\leq\bq\), build
a desired \(\bu\) by appending letters. From \(\bv\bs^2\mathbf U\bs^2\), I at the last
square gives \(\bv\bs^2\mathbf U\bs^2z\bs^2\) up to \(H(\bq)\); B and Corner keep this
intermediate at most \(\bq\). M with initial context \(\bv\) and variables
\(P=\bs,U=\mathbf U,V=z\) gives \(\bv\bs^2\mathbf U z\bs^2\) up to \(H(\bq)\). The one-letter case
is I alone. Thus
\begin{equation}\label{eq:h:019}
\bv\bs^2\leq\bv\bs^2\bu\bs^2+H(\bq),\qquad
\bv\bs^2\bu\bs^2\leq\bv\bs^2
\tag{\text{H-Right}}
\end{equation}
whenever \(c(\bu)\subseteq c(\bs)\). These arguments use only words on \(c(\bq)\).

Fix a cut \(\bt=\bp^2x\bs^2\) of \(\bq\), with absent contexts omitted. Lemma~\ref{lem:lifting-semantics}
provides \(\mathbf f=\bu y\bv\) in \(\mathbf F\) with \(c(\bu)\subseteq c(\bp)\) and
\(c(\bv)\subseteq c(\bs)\). If \(\bu\) is nonempty, apply Left with final context
\(x\bs^2\) or \(x\); if \(\bv\) is nonempty, apply Right with the current
initial context, which contains \(x\). Skip the corresponding insertion
when \(\bu\) or \(\bv\) is absent. An absent \(\bp\) forces \(\bu\) absent, and an
absent \(\bs\) forces \(\bv\) absent. The resulting word \(\mathbf T=\mathbf Lx\mathbf R\) satisfies
\begin{equation}\label{eq:h:020}
\bt\leq\mathbf T+H(\bq),\qquad\mathbf T\leq\bt\leq\bq,
\end{equation}
where \(\mathbf L=\bp^2\bu\bp^2\) if \(\bu\) is present,
\(\mathbf L=\bp^2\) if only \(\bp\) is
present, and is absent otherwise; \(\mathbf R\) is the corresponding suffix.

If \(y=x\), put \(\mathbf T'=\mathbf T\). If \(y\) occurs in \(\bp\), apply I to the last
\(\bp^2\) before \(x\), inserting \(y\) between two such squares. The actual
outer contexts are either nonempty words or absent variants of the finite
pattern. B proves \(\mathbf T'\leq\mathbf T\), and Error proves \(\mathbf T\leq\mathbf T'+H(\bq)\).
The clean block \(\bu\) occurs before the inserted \(y\), and \(\bv\) after it.
If \(y\) occurs in \(\bs\), use I at the first \(\bs^2\) after \(x\); the same
facts hold. One case applies since \(y\in c(\bp)\cup\{x\}\cup c(\bs)\).

Retain the nonempty blocks among \(\bu,y,\bv\) in \(\mathbf T'\), in that order.
There are one, two or three. Treat every nonempty intervening or outer gap
as a context variable and select an absent-context version for every empty
gap. The resulting single C pattern has lower word exactly \(\mathbf T'\) and
unlifted upper word \(\mathbf f\). Because \(\mathbf T'\leq\bq\) was separately proved,
Error gives \(\mathbf T'\leq\mathbf f+H(\bq)\). Thus \(\bt\leq\mathbf F+H(\bq)\). Sum over all cuts
and use Lemma~\ref{lem:cut-expansion} to obtain SL. This also covers a one-letter \(\bq\), when
both cut contexts are absent and C still has its mandatory pivot.
\end{proof}

\paragraph{A finite coordinate calculus.}\label{sec:4}

The next identities deal with the terms arising when a word takes the
value $b$. All their variables are ordinary variables; no word-pattern
enumeration is needed. Put
\begin{equation}\label{eq:h:022}
g=QxQ,\qquad e=Q^2+g^2,\qquad s\otimes t=sxt,
\end{equation}
and abbreviate
\begin{equation}\label{eq:h:023}
\mathsf B(t)=gxtxg,\quad \mathsf R(t)=gxtg,\quad
\mathsf P(t)=gtxg,\quad \mathsf T(t)=gtg.
\end{equation}
For a term $\bd$, write $\mathbf A={}_{\bd}\mathbf B$ for
$\mathbf A+\bd=\mathbf B+\bd$ and
$\mathbf A\leq_{\bd}\mathbf B$ for
$\mathbf A+\mathbf B+\bd=\mathbf B+\bd$.
Let $h_t=gtg$. Each expression in this notation has a fixed expansion
in the original signature.

Let $E_{\rm coord}$ consist of the following seventeen identities.
\begin{enumerate}
\item The three identities
\begin{equation}\label{eq:h:024}
g\otimes g={}_e g,\qquad gx^2g={}_e g,\qquad
\mathsf B(Q)={}_e g.
\end{equation}
\item The four identities
\begin{equation}\label{eq:h:026}
h_t\leq g+e+h_t^2,\qquad
g\otimes h_t={}_e h_t,\qquad
h_t\otimes g={}_e h_t,\qquad
h_t\otimes h_t={}_e h_t.
\end{equation}
\item The four coordinate inequalities
\begin{equation}\label{eq:h:027}
\mathsf T(y)\leq_e\mathsf R(y),\quad
\mathsf T(y)\leq_e\mathsf P(y),\quad
\mathsf R(y)\leq_e\mathsf B(y),\quad
\mathsf P(y)\leq_e\mathsf B(y).
\end{equation}
\item The identity
\begin{equation}\label{eq:h:coordinate-commute}
h_s\otimes h_t={}_{e+h_s^2+h_t^2}h_t\otimes h_s.
\end{equation}
\item The four product formulas
\begin{equation}\label{eq:h:028}
\mathsf B(uv)={}_e
\mathsf R(u)\otimes\mathsf B(v)
+\mathsf B(u)\otimes\mathsf P(v),
\end{equation}
\begin{equation}\label{eq:h:029}
\mathsf R(uv)={}_e
\mathsf R(u)\otimes\mathsf R(v)
+\mathsf B(u)\otimes\mathsf T(v),
\end{equation}
\begin{equation}\label{eq:h:030}
\mathsf P(uv)={}_e
\mathsf T(u)\otimes\mathsf B(v)
+\mathsf P(u)\otimes\mathsf P(v),
\end{equation}
\begin{equation}\label{eq:h:031}
\mathsf T(uv)={}_e
\mathsf T(u)\otimes\mathsf R(v)
+\mathsf P(u)\otimes\mathsf T(v).
\end{equation}
\item The identity $\mathsf B(t)\leq t+e$.
\end{enumerate}

\begin{lemma}\label{lem:coordinate-validity}
Every identity in $E_{\rm coord}$ holds in $\Abar$.
\end{lemma}
\begin{proof}
Fix an assignment, and write $p,r$ for the values of $Q,x$.
If $p=0$, all the guarded expressions vanish.
If $p\in\{a,ab,ba\}$, then $e=p$, while $g$, each coordinate,
each sandwich $h_t$, and their local products belong to $\{0,p\}$.
All the displayed identities therefore hold.
If $p=b$ and $r\ne a$, then $e=g=0$, and again every guarded
expression vanishes.

Suppose that $p=b$ and $r=a$. Then $e=0$ and $g=b$.
Identify $\{0,b\}$ with the two-element Boolean algebra.
Local multiplication $s\otimes t=sat$ is conjunction and addition is
disjunction. Moreover,
\begin{equation}\label{eq:h:032}
\mathsf B(t)=[t(2)\geq1],\quad
\mathsf R(t)=[t(2)=2],\quad
\mathsf P(t)=[t(1)\geq1],\quad
\mathsf T(t)=[t(1)=2].
\end{equation}
The coordinate inequalities follow from monotonicity.
The product formulas are the entries of Boolean multiplication of
\[
\begin{pmatrix}\mathsf R&\mathsf B\\
\mathsf T&\mathsf P\end{pmatrix}.
\]
All the coordinates of $a$ are one, while those of $b$ are
$(1,0,0,0)$ in the order
$(\mathsf B,\mathsf R,\mathsf P,\mathsf T)$.
Also, $\mathsf B(t)\leq t$.
Every $h_t$ is zero or $b$, so $h_t^2=0$.
These observations prove the identities in this case.

It remains to consider $p=1$.
If $r\ne1$, the guarded expressions are absorbed by
$e=1+r^2$: for $r\in\{0,b,ab,ba\}$ they are zero or $r\leq1$,
and for $r=a$ the element $e=a$ is greatest.
If $r=1$, then $e=g=1$, $h_t=t$, $C(t)=t$ for every coordinate,
and $\otimes$ is the original multiplication.
The first group, the unit laws, the coordinate inequalities, the
product formulas, and the last identity are immediate.
The other identities follow from
\[
t\leq1+t^2,\qquad t+1=t^2+1,\qquad
st+ts\leq s^2+t^2.
\]
The first two relations hold because every element other than $b$
is idempotent and $b\leq1$. The last follows from square additivity.
This completes the proof.
\end{proof}

\begin{lemma}\label{lem:coordinate-instantiation}
Let $\mathbf Q$ be a nonempty word and let $x\in c(\mathbf Q)$.
In the coordinate calculus on variables from $c(\mathbf Q)$, AI, Base, and
$E_{\rm coord}$ derive the following properties, with $e_0=\mathbf Q^2$.
\begin{enumerate}
\item $e=e_0$, and $e_0\otimes e_0$, $e_0\otimes g$,
and $g\otimes e_0$ are at most $e_0$.
\item $C(x)={}_{e_0}g$ for every coordinate $C$,
$\mathsf B(\mathbf Q)={}_{e_0}g$, and
$\mathsf R(\mathbf Q),\mathsf P(\mathbf Q),\mathsf T(\mathbf Q)\leq e_0$.
\item For every $y\in c(\mathbf Q)$, the coordinate $C(y)$ is at most
$g$ modulo $e_0$, has $g$ as a local unit modulo $e_0$, and is
locally idempotent modulo $e_0$.
Also, $e_0\otimes C(y),C(y)\otimes e_0\leq e_0$.
\item Any two primitive coordinates commute locally modulo $e_0$.
The four coordinate inequalities, the product formulas, and
$\mathsf B(t)\leq t+e_0$ remain valid with this same error term.
\end{enumerate}
\end{lemma}
\begin{proof}
Lemma~\ref{lem:content-sandwich} gives $g=\mathbf Qx\mathbf Q\leq\mathbf Q$.
Hence $g^2\leq\mathbf Q^2$ and $e=\mathbf Q^2+g^2=\mathbf Q^2=e_0$.
Every primitive coordinate $C(y)$ has the form $\mathbf Q\bv\mathbf Q$, where $\bv$
is a nonempty word on $c(\mathbf Q)$. Thus $C(y)\leq\mathbf Q$ and
$C(y)^2\leq\mathbf Q^2$.

For any such word $\bv$, the same content-contraction lemma gives
\begin{align*}
e_0\otimes(\mathbf Q\bv\mathbf Q)
&=\mathbf Q(\mathbf Qx\mathbf Q)\bv\mathbf Q
\leq\mathbf Q^2\bv\mathbf Q=\mathbf Q(\mathbf Q\bv\mathbf Q)
\leq\mathbf Q^2,\\
(\mathbf Q\bv\mathbf Q)\otimes e_0
&=\mathbf Q\bv(\mathbf Qx\mathbf Q)\mathbf Q
\leq\mathbf Q\bv\mathbf Q^2=(\mathbf Q\bv\mathbf Q)\mathbf Q
\leq\mathbf Q^2.
\end{align*}
Taking $\bv=x$ proves the bounds involving $g$; taking the appropriate
$\bv$ proves the bounds involving each primitive coordinate.
The bound $e_0\otimes e_0=\mathbf Q^2x\mathbf Q^2\leq\mathbf Q^2$ follows by applying
the lemma to the word $\mathbf Q^2$.

The coordinates of $x$ satisfy
\[
\mathsf T(x)=gxg=g\otimes g,\qquad
\mathsf R(x)=\mathsf P(x)=gx^2g,\qquad
\mathsf B(x)=gx^3g=gx^2g.
\]
Their required values follow from \eqref{eq:h:024} and the power law.
The identity for $\mathsf B(\mathbf Q)$ is already in \eqref{eq:h:024}.
For the other three coordinates, calculate
\begin{align*}
\mathsf R(\mathbf Q)&=\mathbf Qxg^2
\leq\mathbf Qx\mathbf Q^2=g\mathbf Q\leq\mathbf Q^2,\\
\mathsf P(\mathbf Q)&=g^2x\mathbf Q
\leq\mathbf Q^2x\mathbf Q=\mathbf Qg\leq\mathbf Q^2,\\
\mathsf T(\mathbf Q)&=\mathbf Qx\mathbf Q^3x\mathbf Q
\leq\mathbf Qx\mathbf Q^2x\mathbf Q=g^2\leq\mathbf Q^2.
\end{align*}

Write each primitive coordinate as $C(y)=g\bt_C(y)g$, where
\[
\bt_{\mathsf B}(y)=xyx,\quad\bt_{\mathsf R}(y)=xy,\quad
\bt_{\mathsf P}(y)=yx,\quad\bt_{\mathsf T}(y)=y.
\]
Substitution in \eqref{eq:h:026} gives the upper bound, the two
unit laws, and local idempotence. The square term in the upper bound
is absorbed by $e_0$ because $C(y)^2\leq\mathbf Q^2$.
For two primitive coordinates, substitute their corresponding
words $\bt_C(y),\bt_D(z)$ in \eqref{eq:h:coordinate-commute}.
Both added squares are at most $e_0$, so the identity becomes
commutativity modulo $e_0$.
Finally, \eqref{eq:h:027}--\eqref{eq:h:031} and the last identity
already have the required form after $e$ is replaced by $e_0$.
All substitutions are nonempty words or terms.
\end{proof}

\begin{lemma}[Guarded coordinates of arbitrary terms]\label{lem:coordinate-terms}
Fix a nonempty word \(\mathbf Q\), \(x\in c(\mathbf Q)\), and any terms \(\mathbf L,\mathbf F\) using only
\(c(\mathbf Q)\). If \(\Abar\models\mathbf L\leq\mathbf F\), AI, Base, and \(E_{\rm coord}\)
derive
\begin{equation}\label{eq:h:033}
C(\mathbf L)\leq C(\mathbf F)+\mathbf Q^2
\qquad(C=\mathsf B,\mathsf R,\mathsf P,\mathsf T).
\tag{\text{H-Coord}}
\end{equation}
In particular, if \(\Abar\models\bq\leq\mathbf F\) is retained, that theory derives
\begin{equation}\label{eq:h:034}
\bq x\bq\leq\mathbf F+\bq^2\qquad(x\in c(\bq)).
\tag{\text{H-GC}}
\end{equation}
\end{lemma}
\begin{proof}
Work in an arbitrary model of AI, Base, and
\(E_{\rm coord}\), and fix an assignment.  By
Lemma~\ref{lem:coordinate-instantiation}, all the coordinate laws needed
below hold at the word \(\mathbf Q\) modulo \(\mathbf Q^2\), since their designated
letters belong to \(c(\mathbf Q)\). Put \(e=\mathbf Q^2,g=\mathbf Qx\mathbf Q\). Let \(A_0\) be the local algebra generated by
\(e,g,C(y)\) for all \(y\in c(\mathbf Q)\) under addition and
\(s\otimes t=sxt\). AI makes \(\otimes\) associative and distributive.

Lemma~\ref{lem:coordinate-instantiation} gives \(e\otimes a\leq e\) and \(a\otimes e\leq e\) for
each generator. Induction extends these inequalities to all elements of
\(A_0\): for a product, for example,
\(e\otimes(a\otimes b)=(e\otimes a)\otimes b\leq e\otimes b\leq e\).
It follows by expansion that
\(a\sim b\iff a+e=b+e\) is a congruence on this local algebra. For example
\((a+e)\otimes(c+e)+e=(a\otimes c)+e\). This does not assert a congruence
on the entire ambient semiring.

In the quotient \(A_0/{\sim}\), the class of \(e\) is zero, the class of
\(g\) is a greatest element and a multiplicative unit, and the primitive
coordinates commute and are idempotent under \(\otimes\). Unit and upper
bounds extend from generators by induction. Commutation and idempotence
extend to products. A product is below each factor because \(g\) is
greatest, so cross products in the square of a sum are absorbed; hence
every sum is idempotent as well. Thus \(+\) and \(\otimes\) are the join
and meet of a bounded distributive lattice. To check the greatest-lower-
bound property explicitly, if \(d\leq a,b\), then
\(d=d\otimes d\leq a\otimes b\), whereas \(a\otimes b\leq a,b\).
The lattice is finite: distribution, commutation and idempotence express
each element as a join of products of subsets of its finite generator set.

For every term \(\bt\) on \(c(\mathbf Q)\), simultaneously prove that there is an
expression \(p_C(\bt)\) in the local generators satisfying
\begin{equation}\label{eq:h:035}
C(\bt)+e=p_C(\bt)+e,\qquad
e\otimes C(\bt)\leq e,\quad C(\bt)\otimes e\leq e.
\tag{\text{H-Expand}}
\end{equation}
For letters this is Lemma~\ref{lem:coordinate-instantiation}. Addition follows by distributivity. At a
product, \eqref{eq:h:028}--\eqref{eq:h:031} express its coordinate modulo \(e\) using the
coordinates of the two factors. The inductive absorption assertions make
it legitimate to replace each such factor by its polynomial modulo
\(e\): expanding \((a+e)\otimes(b+e)\) introduces only absorbed terms.
This gives the required polynomial for the product. The resulting bound
\(C(\bt)\leq p_C(\bt)+e\), multiplied by \(e\) on either local side, proves
the two new absorption assertions. No unproved global quotient step occurs.
The images of the expanded coordinates therefore satisfy the Boolean
matrix formulas for every term, not merely words.

We recall a separation argument for a finite distributive lattice. If
\(u\not\leq v\), choose an ideal maximal among those containing \(v\)
and excluding \(u\). Its complement is a filter: if \(a,b\) are outside,
maximality gives \(u\leq i\vee a\) and \(u\leq j\vee b\) for some
\(i,j\) in the ideal. If \(a\wedge b\) were inside, distributing the meet
of these bounds would put \(u\) inside, a contradiction. The characteristic
function of the complement consequently preserves join and meet, and
sends \(u\) to one and \(v\) to zero.

If the image of \(C(\mathbf L)\) were not below that of \(C(\mathbf F)\), apply this
separation to their expanded polynomials. The four primitive images at
each letter obey \(\mathsf T\leq\mathsf R\leq\mathsf B\) and
\(\mathsf T\leq\mathsf P\leq\mathsf B\).
There are exactly six Boolean quadruples satisfying these inequalities,
namely the coordinate quadruples of the six maps in Lemma~\ref{lem:elementary-structure}. Assign each
letter that actual element of \(\Abar\). The product and addition formulas
show that its actual term coordinates equal the separated images.
Lemma~\ref{lem:coordinate-instantiation} also forces \(x=a,\mathbf Q=b\) in this assignment. Validity of
\(\mathbf L\leq\mathbf F\) and monotonicity of the coordinate functions contradict the
separated values one and zero. Therefore the desired lattice inequality
holds. By Expand this says \(C(\mathbf L)\leq C(\mathbf F)+e\) in the original model.
The model and assignment were arbitrary, so equational completeness
gives the required derivation. This proves Coord.

For \eqref{eq:h:034} take \(\mathbf Q=\bq,\mathbf L=\bq,C=\mathsf B\).
Lemma~\ref{lem:coordinate-instantiation} gives
\(\mathsf B(\bq)+e=g+e\), while the last coordinate identity, with the term
\(\mathbf F\) substituted for its auxiliary variable, says
\(\mathsf B(\mathbf F)\leq\mathbf F+e\). Equation~\eqref{eq:h:033} gives \(g\leq\mathbf F+e\), as required.
\end{proof}

\begin{lemma}[Word absorption up to its square]\label{lem:absorption-up-to-square}
For every retained \(\Abar\)-valid \(\bq\leq\mathbf F\), the fixed finite identities
defined so far derive \(\bq\leq\mathbf F+\bq^2\).
\end{lemma}
\begin{proof}
The absorption also holds in \(S\), so Lemma~\ref{lem:finite-lifting} gives \(\bq\leq\mathbf F+H(\bq)\).
Lemma~\ref{lem:coordinate-terms} gives \(\bq x\bq\leq\mathbf F+\bq^2\) for every \(x\in c(\bq)\). Join these
finitely many consequences and use additive transitivity to obtain the
claim. Only the number of steps depends on \(\bq\).
\end{proof}

\paragraph{A scalar guard for the content.}\label{sec:5}

\paragraph{The scalar identities.}
For the auxiliary band calculus, write \(c=ab,d=ba\) for
the two indicated element labels of the already defined algebra.
These abbreviations do not introduce nullary symbols.

Use the purely typographical abbreviations
\begin{equation}\label{eq:b:001}
h_T(\bu)=T\bu^2T,\qquad \bu\star \bv=(\bu\bv)^2.
\end{equation}
The symbol \(\star\) is not a primitive operation or a globally asserted
meet. Its lattice laws below apply only under a specified common bound.

Let \(E_{\rm BG}\) be AI together with the following fifteen identities.
All inequalities are expanded as addition absorption. Include
\begin{equation}\label{eq:b:002}
x^3=x^2,\qquad x^2\le x,\qquad
(x+y)^2=x^2+y^2,\qquad xy\le x+y.
\tag{\text{B-B0}}
\end{equation}
For the next four identities put \(T=x+z\) and \(H=h_T(x)\):
\begin{equation}\label{eq:b:003}
H\le T,\qquad H^2=H,\qquad TH=H,\qquad HT=H.
\tag{\text{B-N}}
\end{equation}
For each of the next six identities separately, let \(m\) be its largest
displayed subscript, put \(T=x_1+\cdots+x_m+z\), and \(H_i=h_T(x_i)\):
\begin{equation}\label{eq:b:004}
H_1\star H_2=H_2\star H_1,
\tag{\text{B-L1}}
\end{equation}
\begin{equation}\label{eq:b:005}
(H_1\star H_2)\star H_3=H_1\star(H_2\star H_3),
\tag{\text{B-L2}}
\end{equation}
\begin{equation}\label{eq:b:006}
H_1\star H_1=H_1,
\tag{\text{B-L3}}
\end{equation}
\begin{equation}\label{eq:b:007}
H_1+(H_1\star H_2)=H_1,
\tag{\text{B-L4}}
\end{equation}
\begin{equation}\label{eq:b:008}
H_1\star(H_1+H_2)=H_1,
\tag{\text{B-L5}}
\end{equation}
\begin{equation}\label{eq:b:009}
H_1\star(H_2+H_3)=(H_1\star H_2)+(H_1\star H_3).
\tag{\text{B-L6}}
\end{equation}
For the last identity put \(T=x+y+z+t\). Let \(A_T(x,y,z)\) be the
left-associated \(\star\)-product in precisely the following order:
\begin{equation}\label{eq:b:010}
h_T(x),h_T(y),h_T(z),h_T(xy),h_T(xz),h_T(yx),
h_T(yz),h_T(zx),h_T(zy).
\end{equation}
Include the single identity
\begin{equation}\label{eq:b:011}
A_T(x,y,z)\le h_T(xyz).
\tag{\text{B-C}}
\end{equation}
Every abbreviation expands into an ordinary term. These are fixed identities
in at most four variables, not schemes indexed by words. Including AI,
\(E_{\rm BG}\) has 21 displayed identities.

\begin{lemma}[Validity of the scalar identities]\label{band:basic-validity}
Every identity in \(E_{\rm BG}\) is valid in \(\Abar\).
\end{lemma}
\begin{proof}
The square map fixes \(0,1,a,c,d\) and sends \(b\) to zero. This proves
cube stability and square decrease. It preserves order along both displayed
chains, so it preserves joins of comparable elements. The only incomparable
pair is \(c,d\), for which \(c^2+d^2=1=(c+d)^2\). Thus square additivity
holds. Finally \(xy\le(x+y)^2\le x+y\), by monotonicity and square decrease.

For elements \(u\le T\), direct multiplication gives the scalar table:

\begin{center}
\small
\begin{tabular}{p{.12\linewidth}p{.27\linewidth}p{.48\linewidth}}
\hline
\(T\) & possible \(h_T(u)\) & addition and \(\star\) on these values \\
\hline
\(a\) & \(\{0,a\}\) & the two-element lattice \\
\(1\) & \(\{0,c,d,1\}\) & the four-element Boolean lattice with atoms \(c,d\) \\
\(c\) & \(\{0,c\}\) & the two-element lattice \\
\(d\) & \(\{0,d\}\) & the two-element lattice \\
\(b\) or \(0\) & \(\{0\}\) & the one-element lattice \\
\hline
\end{tabular}
\end{center}

For \(T=a\), the inputs \(0,b\) give zero and the other four give \(a\).
For \(T=1\), the scalar is \(u^2\); on the four displayed values the
only nontrivial off-diagonal products are \(cd=0,dc=b\), both of square
zero. This is exactly Boolean meet after squaring. For \(T=c\), the
possible inputs \(0,b,c\) yield \(0,0,c\); the \(d\) case is opposite.
For \(T=b\) or zero, all permissible inputs have square zero. The same
calculations give \(H\le T,H^2=H,TH=H=HT\). Since the displayed bounds
in \eqref{eq:b:003} and each lattice identity in \eqref{eq:b:004}--\eqref{eq:b:009} contain their inputs as join summands, these calculations
prove all normalization and lattice templates.

For \eqref{eq:b:011}, first record a square-content calculation on
\(S=\{0,1,b,c,d\}\). The map \(\rho(s)=s^2\) takes values in
\(\{0,c,d,1\}\) and satisfies
\begin{equation}\label{eq:b:012}
\rho(st)=\rho(s)\star\rho(t),
\tag{\text{B-1}}
\end{equation}
where the right side is Boolean meet. The elements zero and \(1\)
supply the zero and identity cases. If a factor is \(b\), the product
is zero or \(b\), so both sides are zero. The remaining cases are
\(cc=c,dd=d,cd=0,dc=b\), proving \eqref{eq:b:012}. Induction shows that the square
of any word evaluated in \(S\) is the Boolean meet of its squared letters.

All block words in \eqref{eq:b:011} are at most \(T\), using \(uv\le u+v\).
If \(T=a\), the nine scalars belong to \(\{0,a\}\). A zero factor makes
the desired inequality immediate. Otherwise \(x,y,z\) are nonzero
idempotents, and every ordered pair of distinct displayed positions has
nonzero square. Thus no value is \(b\), and \(c,d\) cannot both occur.
The values lie in \(\{1,a,c\}\) or \(\{1,a,d\}\), each a multiplicatively
closed band of nonzero elements. Hence \(xyz\) is a nonzero idempotent
and \(h_a(xyz)=a=A_T\).

If \(T=1\), all values lie in \(S\); by \eqref{eq:b:012} the meet of the first three
factors already equals \(h_1(xyz)\), and the remaining meet factors can
only decrease it. If \(T=c\), a nonzero meet forces \(x=y=z=c\), giving
equality; the \(T=d\) case is the same. For \(T=b\) or zero, the meet
is zero. This proves \eqref{eq:b:011}.
\end{proof}

\begin{lemma}[Scalar calculus under a common bound]\label{band:scalar-calculus}
Fix a term \(T\). For terms \(\bu\) with a derived inequality \(\bu\le T\),
every expression generated from \(h_T(\bu)\) by addition and \(\star\)
is derivably representable as \(h_T(\bv)\) with \(\bv\le T\). These
expressions form a distributive lattice in the equational calculus,
with addition as join and \(\star\) as meet. Every such expression \(H\)
satisfies
\begin{equation}\label{eq:b:013}
H\le T,\qquad H^2=H,\qquad TH=H=HT.
\tag{\text{B-2}}
\end{equation}
\end{lemma}
\begin{proof}
A derived bound is implemented by substitution, not a new inference rule.
Given \(\bu\le T\), substitute \(x=\bu,z=T\) in \eqref{eq:b:003}; the displayed bound
\(\bu+T\) is derivably \(T\). For several bounds \(\bu_i\le T\), substitute
\(x_i=\bu_i,z=T\) in the lattice list \eqref{eq:b:004}--\eqref{eq:b:009}. In \eqref{eq:b:011}, substitute its three block variables by
bounded terms and its fourth variable by \(T\). Again its displayed sum
equals \(T\). All substitutions use nonempty terms.

For addition, square additivity gives
\begin{equation}\label{eq:b:014}
h_T(\bu)+h_T(\bv)=T(\bu^2+\bv^2)T=h_T(\bu+\bv),
\end{equation}
and \(\bu+\bv\le T\). For \(\star\), let \(H=h_T(\bu),K=h_T(\bv)\), and put
\(\bw=HK\). Normalization gives \(\bw\le H+K\le T\) and \(T\bw=\bw=\bw T\).
Consequently
\begin{equation}\label{eq:b:015}
h_T(\bw)=T\bw^2T=(T\bw)(\bw T)=\bw^2=H\star K.
\end{equation}
Structural induction proves representability. Each expression can now be
inserted into the lattice list \eqref{eq:b:004}--\eqref{eq:b:009} after this representation, so all six lattice laws apply
under the common bound. Normalization gives \eqref{eq:b:013}.

Explicitly, absorption and commutativity give \(H\star K\le H,K\).
If another scalar \(J\) satisfies \(J\le H,K\), monotonicity gives
\(J=J\star J\le H\star K\). Induction gives the greatest-lower-bound
property for finite meets. These claims apply only to scalar expressions
with this common \(T\).
\end{proof}

\paragraph{Literal content guard.}
Order a finite nonempty alphabet \(C=\{x_1,\ldots,x_m\}\), and put
\(T_C=x_1+\cdots+x_m\). List first \(x_1,\ldots,x_m\), then every
\(x_ix_j\) in lexicographic order of index pairs, including \(i=j\).
Call the list \(R_C\). Define \(B_C\) to be the left-associated
\(\star\)-product of \(h_{T_C}(\br)\), with \(\br\) in that order.
This is a literal nonempty term with content exactly \(C\). The scalar
calculus proves independence of ordering and parenthesization, although
the definition fixes both.

\begin{lemma}[Support of the content guard]\label{band:guard-support}
For every finite nonempty \(C\), the fixed list \(E_{\rm BG}\) derives
\begin{equation}\label{eq:b:016}
B_C^2=B_C,\quad B_C\le T_C,\quad T_CB_C=B_C=B_CT_C,
\tag{\text{B-3}}
\end{equation}
and, for every constant-free term \(\bw\) on \(C\),
\begin{equation}\label{eq:b:017}
B_C\bw B_C=B_C.
\tag{\text{B-4}}
\end{equation}
\end{lemma}
\begin{proof}
The inequality \(uv\le u+v\) implies by induction that every word is
at most the sum of its letters. Thus every word on \(C\) is at most
\(T_C\), and scalar normalization proves \eqref{eq:b:016}.

Write \(T=T_C,B=B_C\). We first derive
\begin{equation}\label{eq:b:018}
B\le h_T(\bw)
\tag{\text{B-5}}
\end{equation}
for every nonempty word \(\bw\) on \(C\). For lengths one and two this
is a defining meet factor. For length \(n\ge3\), write \(\bw=x\bv z\)
with endpoint letters \(x,z\) and nonempty middle word \(\bv\).
All nine words
\begin{equation}\label{eq:b:019}
x,\ \bv,\ z,\ x\bv,\ xz,\ \bv x,\ \bv z,\ zx,\ z\bv
\end{equation}
have length less than \(n\). Induction bounds \(B\) by their nine
scalars and hence by their meet. Apply \eqref{eq:b:011} with block substitutions
\(x\mapsto x,y\mapsto \bv,z\mapsto z,t\mapsto T\); its bound
\(x+\bv+z+T\) equals \(T\), proving \eqref{eq:b:018}. Repeated endpoint letters cause
no exception because substitutions may identify variables.

Multiply \eqref{eq:b:018} by \(B\) on both sides and use \eqref{eq:b:016} and square decrease:
\begin{equation}\label{eq:b:020}
B=B^3\le B(T\bw^2T)B=B\bw^2B\le B\bw B.
\end{equation}
Conversely \(\bw\le T\) gives \(B\bw B\le BTB=B\). This proves \eqref{eq:b:017}
for words. Expand a term into a nonempty sum of words and distribute;
each sandwich is \(B\), so their sum is \(B\). In particular
\(B_CB_DB_C=B_C\) whenever \(\varnothing\ne D\subseteq C\), because
\(B_D\) is a term on \(C\).
\end{proof}

\begin{lemma}[Semantic interpretation of the selector]\label{band:semantic-selector}
Under an \(\Abar\) assignment, \(B_C\ne0\) exactly when the values assigned to
\(C\) are nonzero and generate a multiplicative band (a semigroup of
idempotents). Such values lie in one of
\begin{equation}\label{eq:b:021}
B_R=\{0,c,1,a\},\qquad B_L=\{0,d,1,a\}.
\end{equation}
Call an assignment proper-generation if the subsemiring generated by its
values on \(C\), without adjoining constants, is proper in \(\Abar\). For a word
\(\bq\) with \(c(\bq)=C\) and \(e=\bq^2\),
\begin{equation}\label{eq:b:022}
eB_Ce=
\begin{cases}
e,&\text{on proper-generation assignments},\\
0,&\text{on assignments generating all of }\Abar.
\end{cases}
\tag{\text{B-6}}
\end{equation}
\end{lemma}
\begin{proof}
The table and order show that \(S=\{0,1,b,c,d\}\), \(B_R\), and \(B_L\)
are subsemirings. If a generator set has no \(a\), it lies in \(S\).
If it contains \(a,b\), it generates \(c=ab,d=ba,0=b^2,1=c+d\),
hence the whole algebra. If it contains \(a,c,d\), it generates
\(b=dc,0=cd,1=c+d\), again the whole algebra. In the remaining case
with \(a\) present, \(b\) is absent and at least one of \(c,d\) is
absent, so the generator set lies in \(B_R\) or \(B_L\). Thus every
proper generator set lies in one of the three displayed proper subsemirings.

A nonzero scalar meet \(B_C\) forces every letter and ordered-pair
scalar to be nonzero. A value zero or \(b\) makes its letter scalar
zero. Simultaneous values \(c,d\) make the scalar for the ordered
product \(cd=0\) zero. The remaining values lie in \(\{1,a,c\}\) or
\(\{1,a,d\}\), both nonzero bands. Conversely, if the nonzero
generators generate a band, none is \(b\), and \(c,d\) cannot both
occur because \(dc=b\) is not idempotent. If \(a\) occurs, all scalar
factors are \(a\), and \(B_C=a\). Otherwise the generators lie in
\(\{1,c\}\), \(\{1,d\}\), or \(\{1\}\). The scalar meet is \(c\)
if \(c\) occurs, \(d\) if \(d\) occurs, and \(1\) otherwise.
This proves the characterization. Notice that \(cd=0\) alone would
not contradict being a band; the nonidempotent product \(dc=b\) is
the needed observation in this converse.

For \eqref{eq:b:022}, the case \(e=0\) is immediate. Suppose \(e\ne0\) and generation
is proper. Every variable of \(\bq\) occurs, so no assigned value is zero.
In the \(S\) case, \eqref{eq:b:012} expresses \(e\) as the Boolean meet of squared
generator values. A nonzero meet excludes \(b\) and the simultaneous
pair \(c,d\); the values lie in \(\{1,c\}\), \(\{1,d\}\), or
\(\{1\}\). Then \(\bq=e=B_C\) equals respectively \(c,d,1\), so
\(eB_Ce=e\). In a band case, with no \(a\), the same argument applies.
If \(a\) occurs, then \(B_C=a\). The word \(\bq=e\) has value in
\(\{a,c\}\) for \(B_R\) or \(\{a,d\}\) for \(B_L\): in these
bands a word containing \(a\) cannot have value \(1\). The table gives
\(aaa=a,cac=c,dad=d\), so \(eB_Ce=e\).
For full generation, the classification forces \(a,b\), or \(a,c,d\),
among the generators. A letter scalar in the first case, or an ordered-pair
scalar in the second, is zero. Hence \(B_C=0\).
\end{proof}

\paragraph{The auxiliary band calculus.}\label{sec:6}

\paragraph{Band basis and occurrence profiles.}
Let \(E_{\rm band}\) be AI together with
\begin{equation}\label{eq:b:023}
x^2=x,
\tag{\text{B-I}}
\end{equation}
\begin{equation}\label{eq:b:024}
xyxzx=xyzx,
\tag{\text{B-RB}}
\end{equation}
\begin{equation}\label{eq:b:025}
x+y+xyz=x+y+yxz,
\tag{\text{B-ER}}
\end{equation}
\begin{equation}\label{eq:b:026}
x+y+zxy=x+y+zyx.
\tag{\text{B-EL}}
\end{equation}
For a word \(\bw\), let \(\operatorname{first}(\bw)\) retain the first
occurrence of every letter and \(\operatorname{last}(\bw)\) retain the
last occurrence of every letter, preserving their orders. Both are nonempty.
For \(x\in c(\bq)\), put
\begin{equation}\label{eq:b:027}
R_\bq(x)=\{y:\text{the last }y\text{ in }\bq
                  \text{ occurs after its last }x\},
\end{equation}
\begin{equation}\label{eq:b:028}
L_\bq(x)=\{y:\text{the first }y\text{ in }\bq
                  \text{ occurs before its first }x\}.
\end{equation}

\begin{lemma}[Band validity and word normal form]\label{band:band-validity-and-normal-form}
The identities in \(E_{\rm band}\) hold in \(B_R\times B_L\).
From \eqref{eq:b:023} and \eqref{eq:b:024} one can derive, for every nonempty word \(\bw\),
\begin{equation}\label{eq:b:029}
\bw=\operatorname{first}(\bw)\operatorname{last}(\bw).
\tag{\text{B-7}}
\end{equation}
An identity between words holds in \(B_R\times B_L\) exactly when their
first lists and their last lists both agree. In particular, for nonempty
words \(\bp,\bq\) with \(c(\bp)\subseteq c(\bq)\),
\begin{equation}\label{eq:b:030}
\bq\bp\bq=\bq
\tag{\text{B-8}}
\end{equation}
is derivable from \eqref{eq:b:023} and \eqref{eq:b:024}.
\end{lemma}
\begin{proof}
Addition in \(B_R\) is maximum for \(0<c<1<a\). Zero absorbs
multiplication, \(1\) is an identity, and \(c,a\) are right zeros
among the nonzero elements. Thus a word with a zero factor has value
zero; otherwise it has the value of its last factor unequal to \(1\),
or \(1\) if every factor is \(1\). The algebra \(B_L\) is its
multiplicative opposite under \(c\mapsto d\), so its rule uses the
first factor unequal to \(1\). These rules follow from the table.

Idempotence \eqref{eq:b:023} follows. Both sides of \eqref{eq:b:024} have the same first list
\(xyz\) and last list \(yzx\), so both evaluation rules give equality.
For \eqref{eq:b:025}, any zero among \(x,y,z\) makes both products zero.
Suppose they are nonzero. In \(B_R\), if \(z=c\) or \(a\), both
products equal \(z\). If \(z=1\), each of \(xy,yx\) is at most
\(x+y\), since \(xy\le(x+y)^2=x+y\). Thus the two sums agree.
In \(B_L\), if \(x=y=1\), both products equal \(z\). Otherwise
each product is one of the nonidentity values among \(x,y\), so each
is at most \(x+y\). This proves \eqref{eq:b:025}. Opposites give \eqref{eq:b:026} in both
factors.

An internal occurrence of a letter can be deleted if the same letter
has a retained occurrence on both sides. In a factor \(x\bu x\bv x\),
if both intervening words \(\bu,\bv\) are nonempty, substitute them for
the other two variables in \eqref{eq:b:024}. If \(\bu\) is absent, apply \eqref{eq:b:023} to
the first adjacent \(xx\); if \(\bv\) is absent, use the last adjacent
pair. If both are absent, \(xxx=xx\) deletes the middle occurrence.
In the adjacent-pair cases retain the flanking occurrence. These are
separate applications, never empty substitutions.

Substitute \(x=\bw\) in \eqref{eq:b:023}, obtaining \(\bw=\bw\bw\). Mark the first
occurrence of each letter in the first copy and the last occurrence
in the second copy. Every unmarked occurrence has its marked occurrences
on either side. Delete all unmarked occurrences with the preceding rule,
preserving the marks. The result is precisely \eqref{eq:b:029}. Thus equality of both
lists implies a derivable equality.

Conversely, different contents are separated by assigning a variable
present on only one side to zero and all others to \(1\). If contents
agree and first lists differ, the two permutations of the content contain
a pair in reversed relative order. Assign the pair values \(a,d\) in
\(B_L\), all other letters \(1\); the first-nonidentity rule separates
the words. A reversed pair in the last lists is likewise separated by
values \(a,c\) in \(B_R\). This proves the word criterion. Finally
\(\bq\bp\bq\) and \(\bq\) have the same first and last lists when
\(c(\bp)\subseteq c(\bq)\), so \eqref{eq:b:029} derives \eqref{eq:b:030}.
\end{proof}

\begin{lemma}[The band cover criterion]\label{band:polynomial-cover-criterion}
Let \(\bq\) be a nonempty word and \(\mathbf{P}\) a finite nonempty polynomial
whose list of word summands is \(F\). Then \(\bq\le \mathbf{P}\) holds in
\(B_R\times B_L\) if and only if, for every \(x\in c(\bq)\), both
of the following witnesses exist:

\begin{enumerate}
\item A summand \(\mathbf{f}_x\in F\), with \(c(\mathbf{f}_x)\subseteq c(\bq)\) and
\(x\in c(\mathbf{f}_x)\), such that every letter after its last \(x\)
belongs to \(R_\bq(x)\).
\item A summand \(\mathbf{g}_x\in F\), with \(c(\mathbf{g}_x)\subseteq c(\bq)\) and
\(x\in c(\mathbf{g}_x)\), such that every letter before its first \(x\)
belongs to \(L_\bq(x)\).
\end{enumerate}
\end{lemma}
\begin{proof}
First work in \(B_R\). Fix \(x\in c(\bq)\), assign \(x\) the value
\(a\), the letters whose last occurrences in \(\bq\) are earlier than
its last \(x\) the value \(c\), the letters of \(R_\bq(x)\) the
value \(1\), and all variables outside \(c(\bq)\) zero. Then \(\bq=a\).
If \(\bq\le \mathbf{P}\), at least one summand is \(a\), because addition is
maximum in a chain. Such a summand contains no outside variable,
contains \(x\) (the only variable valued \(a\)), and has only
\(1\)-valued variables after its last \(x\). This is exactly a
right witness.

Conversely suppose right witnesses exist. Under any \(B_R\) assignment,
if \(\bq=0\) there is nothing to prove. Otherwise its letters, and therefore
all selected witnesses, are nonzero. If \(\bq=c\), any witness is at
least \(c\). If \(\bq=1\), all letters of \(\bq\) and all witnesses have
value \(1\). If \(\bq=a\), let \(x\) be the letter at the last position
in \(\bq\) whose value differs from \(1\). That value is \(a\).
Every letter in \(R_\bq(x)\) has value \(1\), so \(\mathbf{f}_x=a\).
Thus \(\bq\le \mathbf{P}\) holds in every case. This proves the right criterion.
Reverse words and use opposite multiplication to obtain the left
criterion for \(B_L\).

Evaluation in the product is coordinatewise, so the conjunction of the
two criteria gives the result. This also deals with polynomials having
variables outside \(c(\bq)\), since necessity set those variables to zero.
\end{proof}

\begin{lemma}[Completeness of the band calculus]\label{band:band-polynomial-completeness}
Every polynomial absorption \(\bq\le \mathbf{P}\) valid in \(B_R\times B_L\)
is derivable from \(E_{\rm band}\). Consequently \(E_{\rm band}\)
is a finite identity basis of \(B_R\times B_L\).
\end{lemma}
\begin{proof}
Choose one right witness \(\mathbf{f}_x\) and left witness \(\mathbf{g}_x\) for every
letter of \(\bq\). Let \(\mathbf{U}\) be a product of the selected right witnesses
in a fixed order. Every nonempty product of witnesses is at most a
positive power of \(\mathbf{P}\), and \eqref{eq:b:023} with the term substitution \(x=\mathbf{P}\)
reduces that power to \(\mathbf{P}\).

For witnesses \(\bu,\bv\) and a nonempty suffix \(\br\), \eqref{eq:b:025} and the
bounds \(\bu,\bv\le \mathbf{P}\) give
\begin{equation}\label{eq:b:031}
\mathbf{P}+\bu\bv\br=\mathbf{P}+\bv\bu\br.
\tag{\text{B-9}}
\end{equation}
If a nonempty product \(\mathbf{H}\) of preceding witnesses is present, multiply
the \eqref{eq:b:025} instance on the left by \(\mathbf{H}\). Its added terms \(\mathbf{H}\bu,\mathbf{H}\bv\)
are products of witnesses and hence at most \(\mathbf{P}\), giving
\begin{equation}\label{eq:b:032}
\mathbf{P}+\mathbf{H}\bu\bv\br=\mathbf{P}+\mathbf{H}\bv\bu\br.
\tag{\text{B-10}}
\end{equation}
If \(\mathbf{H}\) is absent use \eqref{eq:b:031}. In an adjacent exchange the substituted
\(\br\) consists of the remaining factors followed by the specified
nonempty suffix, so it is always a legal nonempty word. Adjacent
transpositions consequently permit any permutation before such a suffix.

Write the original unreduced word \(\bq=q_1\cdots q_n\). We derive by
descending induction
\begin{equation}\label{eq:b:033}
\mathbf{U}q_i\cdots q_n\le \mathbf{P}.
\tag{\text{B-11}}
\end{equation}
For \(i=n\), put \(x=q_n\). Since \(R_\bq(x)\) is empty, the witness
\(\mathbf{f}_x\) ends in \(x\). Move it to the last witness position modulo
\(\mathbf{P}\), with suffix \(x\), and use \(\mathbf{f}_xx=\mathbf{f}_x\), a consequence of
\eqref{eq:b:023}. The resulting product of witnesses is at most \(\mathbf{P}\).

For \(i<n\), put \(x=q_i\), \(\mathbf{T}=q_{i+1}\cdots q_n\), which is
nonempty. The words \(\mathbf{f}_xx\mathbf{T},\mathbf{f}_x\mathbf{T}\) have the same first list because
\(x\) already occurs in \(\mathbf{f}_x\). Their last lists also agree.
If \(x\in c(\mathbf{T})\), its inserted occurrence is not final. Otherwise
every letter after the last \(x\) in \(\mathbf{f}_x\) lies in
\(R_\bq(x)\subseteq c(\mathbf{T})\). The final occurrences of all those letters
are in \(\mathbf{T}\). Thus, in this latter case, both last lists consist of
the last occurrences from \(\mathbf{f}_x\) outside \(c(\mathbf{T})\), ending in \(x\),
followed by \(\operatorname{last}(\mathbf{T})\). Equation \eqref{eq:b:029} therefore derives
\begin{equation}\label{eq:b:034}
\mathbf{f}_xx\mathbf{T}=\mathbf{f}_x\mathbf{T}.
\tag{\text{B-12}}
\end{equation}
Move \(\mathbf{f}_x\) to the last witness position modulo \(\mathbf{P}\), apply \eqref{eq:b:034},
then restore the witness order using the nonempty suffix \(\mathbf{T}\).
This gives \(\mathbf{P}+\mathbf{U}x\mathbf{T}=\mathbf{P}+\mathbf{U}\mathbf{T}=\mathbf{P}\), by induction. Thus \(\mathbf{U}\bq\le \mathbf{P}\).

Reversal preserves the axiom set: \eqref{eq:b:024} reverses to its instance with
the other two variables interchanged, and \eqref{eq:b:025},\eqref{eq:b:026} exchange.
Apply the preceding derivation to the reversed word and reversed left
witnesses. Reversing the result gives \(\bq\mathbf{V}\le \mathbf{P}\), where \(\mathbf{V}\) is
a product of one selected left witness per letter, in a suitable order.
Since \(c(\mathbf{U}),c(\mathbf{V})\subseteq c(\bq)\), equation \eqref{eq:b:030} gives
\begin{equation}\label{eq:b:035}
\bq=\bq\mathbf{V}\mathbf{U}\bq=(\bq\mathbf{V})(\mathbf{U}\bq)\le \mathbf{P}^2=\mathbf{P}.
\end{equation}
This is the required derivation.

For an arbitrary valid identity expand both sides into polynomials.
Every summand of either side is semantically at most the opposite
polynomial. The result just proved derives each absorption; adding them
gives both inequalities between the polynomials and hence their equality.
Thus the finite list is complete.
\end{proof}

\paragraph{Localization to the proper modes.}\label{sec:7}

\begin{lemma}[Full-content reflection]\label{band:full-content-reflection}
Fix \(C\ne\varnothing\), and put \(B=B_C\). If \(\bu\) is a word with
\(c(\bu)=C\), and \(\bv\) is any term on \(C\), then \(E_{\rm BG}\) derives
\begin{equation}\label{eq:b:038}
\bu\bv B=\bu B,\qquad B\bv\bu=B\bu.
\tag{\text{B-14}}
\end{equation}
Consequently, for nonempty words \(\mathbf{U},\mathbf{V}\) on \(C\), and \(c(\bq)=C\),
\begin{equation}\label{eq:b:039}
\bq^2B\bq^2=\bq B\bq=\bq\mathbf{V}^2B\mathbf{U}^2\bq.
\tag{\text{B-15}}
\end{equation}
If \(c(\mathbf{f})=c(\mathbf{g})=C\), then
\begin{equation}\label{eq:b:040}
\mathbf{f}B\mathbf{g}\le \mathbf{f}\mathbf{g}.
\tag{\text{B-16}}
\end{equation}
If every summand of polynomials \(\mathbf{P}_*,\mathbf{Q}_*\) has content exactly \(C\),
then
\begin{equation}\label{eq:b:041}
\mathbf{Q}_*B\mathbf{P}_*\le \mathbf{Q}_*\mathbf{P}_*.
\tag{\text{B-17}}
\end{equation}
\end{lemma}
\begin{proof}
Support and \eqref{eq:b:036} give both directions
\begin{equation}\label{eq:b:042}
\bu\bv B=\bu\bv B\bu B=\bu(\bv B)\bu B\le \bu B,\qquad
\bu B=\bu B\bu\bv B=(\bu B\bu)\bv B\le \bu\bv B.
\end{equation}
The support instances inserted here are \(B\bu B=B\) and \(B\bu\bv B=B\).
For the other endpoint, without any assumption that \(E_{\rm BG}\)
is invariant under reversal, calculate directly
\begin{equation}\label{eq:b:043}
B\bv\bu=B\bu B\bv\bu=B\bu(B\bv)\bu\le B\bu,\qquad
B\bu=B\bv\bu B\bu=B\bv(\bu B\bu)\le B\bv\bu.
\end{equation}
This proves \eqref{eq:b:038}. Apply it with \(\bv=\bq\) to erase the duplicated
endpoint copies in \(\bq^2B\bq^2\), obtaining \(\bq B\bq\). With \(\bv=\mathbf{V}^2\)
at the left and \(\bv=\mathbf{U}^2\) at the right it gives the other equality
in \eqref{eq:b:039}.

For \eqref{eq:b:040}, support for \(\mathbf{f}\mathbf{g}\) gives
\begin{equation}\label{eq:b:044}
\mathbf{f}B\mathbf{g}=\mathbf{f}(B\mathbf{f}\mathbf{g}B)\mathbf{g}=(\mathbf{f}B\mathbf{f})(\mathbf{g}B\mathbf{g})\le \mathbf{f}\mathbf{g}.
\end{equation}
Each use of \eqref{eq:b:036} is legal because \(B\) is a term on the full
content of that endpoint word. Distributively expand \(\mathbf{Q}_*B\mathbf{P}_*\)
and apply \eqref{eq:b:040} to each pair of summands, obtaining \eqref{eq:b:041}.
Every-summand full content is an essential hypothesis here; no
inequality \(\mathbf{P}B\mathbf{P}\le \mathbf{P}\) for arbitrary polynomials is being used.
\end{proof}

\paragraph{Eight explicit guarded identities.}
Use the ordered alphabets \(\{x\}\) and \(\{x,y,z\}\) in the literal
guard definition, and abbreviate
\begin{equation}\label{eq:b:045}
G_1=B_{\{x\}},\qquad G_3=B_{\{x,y,z\}}.
\end{equation}
The scalar factors of \(G_1\) are indexed by \(x,xx\).
Those of \(G_3\) are indexed, in order, by
\begin{equation}\label{eq:b:046}
x,y,z,xx,xy,xz,yx,yy,yz,zx,zy,zz.
\end{equation}
Let \(t\) be distinct from \(x,y,z\). The set GB consists of exactly
the following eight identities:
\begin{equation}\label{eq:b:047}
G_1tx^2=G_1tx,\qquad x^2tG_1=xtG_1,
\tag{\text{B-GB-I}}
\end{equation}
\begin{equation}\label{eq:b:048}
G_3t(xyxzx)=G_3t(xyzx),\qquad
(xyxzx)tG_3=(xyzx)tG_3,
\tag{\text{B-GB-RB}}
\end{equation}
\begin{equation}\label{eq:b:049}
G_3t(x+y+xyz)=G_3t(x+y+yxz),
\tag{\text{B-GB-ER-L}}
\end{equation}
\begin{equation}\label{eq:b:050}
(x+y+xyz)tG_3=(x+y+yxz)tG_3,
\tag{\text{B-GB-ER-R}}
\end{equation}
\begin{equation}\label{eq:b:051}
G_3t(x+y+zxy)=G_3t(x+y+zyx),
\tag{\text{B-GB-EL-L}}
\end{equation}
\begin{equation}\label{eq:b:051R}
(x+y+zxy)tG_3=(x+y+zyx)tG_3.
\tag{\text{B-GB-EL-R}}
\end{equation}
Put \(E_{\rm prop}=E_{\rm BG}\cup{\rm GB}\). After expansion of the
fixed abbreviations, this is a literal list of 29 identities including
the six AI identities. The eight new identities have two or four
variables, including \(t\). The band-axiom variable set alone has
size one or three. GB is not a scheme over arbitrary band identities.

\begin{lemma}[Guarded lifting through arbitrary contexts]\label{band:guarded-lifting}
Every identity in GB is valid in \(\Abar\). If \(\mathbf{h},\mathbf{k}\) are terms on a finite
nonempty alphabet \(C\), and \(E_{\rm band}\) derives \(\mathbf{h}=\mathbf{k}\), then
\(E_{\rm p\br op}\) derives
\begin{equation}\label{eq:b:053}
B_C\mathbf{h}=B_C\mathbf{k},\qquad \mathbf{h}B_C=\mathbf{k}B_C.
\tag{\text{B-18}}
\end{equation}
The analogous guarded conclusions hold for derived inequalities.
\end{lemma}
\begin{proof}
For a band axiom \(\boldsymbol{\ell}=\br\), its guard is \(B_D\), where
\(D=\{x\}\) for \eqref{eq:b:023} and \(D=\{x,y,z\}\) for the other three.
Under an \(\Abar\) assignment, either \(B_D=0\), making both sides of
each associated GB identity zero, or its variables have nonzero
values contained in a band \(B_\mathbf{R}\) or \(B_\mathbf{L}\). Band-axiom validity
then gives \(\boldsymbol{\ell}=\br\) on those values. The fresh \(t\) is unrestricted,
so both guarded identities hold.

For lifting, choose \(x_0\in C\). If a finite \(E_{\rm band}\)
derivation introduces auxiliary variables, replace every variable
outside \(C\) throughout the derivation by \(x_0\), fixing \(C\).
This substitution fixes the endpoints \(\mathbf{h},\mathbf{k}\) and preserves all
inferences. Thus every term and every axiom-instance substitution in
the derivation may be assumed to use only \(C\). The substitution
uses a one-letter term, never a constant or an empty word.

Put \(B=B_C\). AI steps are already valid globally. For an instance
\(\boldsymbol{\ell}\sigma=\br\sigma\) of a non-AI band axiom, put \(G=B_D\sigma\).
It is a nonempty term on \(C\), so support gives \(BGB=B\).
In a multiplicative context with left and right parts \(\mathbf{L},\mathbf{R}\),
each independently permitted to be absent, calculate
\begin{equation}\label{eq:b:054}
B\mathbf{L}(\boldsymbol{\ell}\sigma)\mathbf{R}
 =BGB\mathbf{L}(\boldsymbol{\ell}\sigma)\mathbf{R}
 =BGB\mathbf{L}(\br\sigma)\mathbf{R}
 =B\mathbf{L}(\br\sigma)\mathbf{R}.
\tag{\text{B-19}}
\end{equation}
The middle equality uses the left GB identity with \(t=B\mathbf{L}\) when
\(\mathbf{L}\) is present and \(t=B\) otherwise, carrying the external
left \(B\) and right context along. For a right guard use
\begin{equation}\label{eq:b:055}
\mathbf{L}(\boldsymbol{\ell}\sigma)\mathbf{R}B
 =\mathbf{L}(\boldsymbol{\ell}\sigma)\mathbf{R}BGB
 =\mathbf{L}(\br\sigma)\mathbf{R}BGB
 =\mathbf{L}(\br\sigma)\mathbf{R}B,
\tag{\text{B-20}}
\end{equation}
where \(t=\mathbf{R}B\) when \(\mathbf{R}\) is present and \(t=B\) otherwise,
carrying the final external \(B\). Thus no missing context is
substituted for any variable.

An arbitrary one-hole semiring context expands distributively as a
sum of unaffected terms and terms \(\mathbf{L}[{\rm \mathbf{h}ole}]\mathbf{R}\), with each
external part absent or nonempty. This follows by induction on the
context: addition adjoins summands, while multiplying by a fixed term
distributes over its polynomial expansion and appends a nonempty word
to the appropriate side of each affected term. Multiplying the whole
sum by the guard and distributing lets \eqref{eq:b:054} or \eqref{eq:b:055} apply to each
affected summand. Addition congruence restores the context. Thus every
axiom-instance-and-context step lifts, and transitivity proves \eqref{eq:b:053}.
For inequalities apply this to the identity \(\mathbf{h}+\mathbf{k}=\mathbf{k}\), then distribute.
\end{proof}

\begin{lemma}[Full-content side covers]\label{band:full-content-side-covers}
Suppose \(\bq\le \mathbf{P}\) is valid in \(B_\mathbf{R}\times B_L\), with \(\bq\) a
nonempty word. Put \(C=c(\bq)\), \(m=|C|\). Choose the witnesses
\(\mathbf{f}_x,\mathbf{g}_x\) from the polynomial-cover criterion, indexed once by
each \(x\in C\), even when two selected word terms coincide.
Let \(\mathbf{U}\) be the product of the right witnesses in any fixed order.
Choose an order of reversed left witnesses for the reversed
right-side construction, and let \(\mathbf{V}\) be the reversal of that
product, so \(\mathbf{V}\) is a product of the original left witnesses.

Let \(\mathbf{P}_*\) be the sum of every product of \(j\) indexed right
witnesses, for \(m\le j\le2m\), whose index sequence contains each
of the \(m\) indices at least once. Define \(\mathbf{Q}_*\) in the same way
from left witnesses. Both are finite nonempty polynomials and every
summand has content exactly \(C\). Moreover
\begin{equation}\label{eq:b:056}
E_{\rm BG}\vdash \mathbf{P}_*\le \mathbf{P},\quad \mathbf{Q}_*\le \mathbf{P},
\tag{\text{B-21}}
\end{equation}
\begin{equation}\label{eq:b:057}
E_{\rm band}\vdash \mathbf{U}^2\bq\le \mathbf{P}_*,\quad \bq\mathbf{V}^2\le \mathbf{Q}_*.
\tag{\text{B-22}}
\end{equation}
The derivations in \eqref{eq:b:057} use only variables of \(C\).
\end{lemma}
\begin{proof}
Every witness uses only \(C\) and contains its indexed letter.
A product containing every index therefore has content exactly \(C\).
The sums are finite because lengths are at most \(2m\), and nonempty
because listing every index once is allowed. Each witness is at most
\(\mathbf{P}\), so a product of \(j\) witnesses is at most \(\mathbf{P}^j\).
Square decrease gives \(\mathbf{P}^j\le \mathbf{P}\) for every \(j\ge1\):
for \(j=2\) use it directly; for \(j>2\), multiply \(\mathbf{P}^2\le \mathbf{P}\)
by the nonempty outer power \(\mathbf{P}^{j-2}\), and iterate. Adding the
resulting bounds proves \eqref{eq:b:056}. This \(\Abar\)-valid argument does not use
multiplicative idempotence of \(\mathbf{P}\).

For \eqref{eq:b:057}, retain the first copy of \(\mathbf{U}\) unchanged in \(\mathbf{U}\mathbf{U} \bt\),
where \(\bt\) is a nonempty suffix, and permute only the factors of
the second copy. For an adjacent pair \(\mathbf{f}_i,\mathbf{f}_j\), let \(\mathbf{H}\)
be the preceding factors in that copy and let \(\mathbf{R}\) be all remaining
factors followed by \(\bt\), so \(\mathbf{R}\) is nonempty. Substitute
\(x=\mathbf{f}_i,y=\mathbf{f}_j,z=\mathbf{R}\) in \eqref{eq:b:025} and multiply on the left by \(\mathbf{U}\mathbf{H}\),
or by \(\mathbf{U}\) if \(\mathbf{H}\) is absent. This exchanges
\(\mathbf{U}\mathbf{H}\mathbf{f}_i\mathbf{f}_j\mathbf{R}\) and \(\mathbf{U}\mathbf{H}\mathbf{f}_j\mathbf{f}_i\mathbf{R}\) with added terms \(\mathbf{U}\mathbf{H}\mathbf{f}_i,\mathbf{U}\mathbf{H}\mathbf{f}_j\).
Their initial \(\mathbf{U}\) contains every index. If \(\mathbf{H}\) has \(h\)
factors, then \(0\le h\le m-2\), so each added term has
\(m+h+1\) factors, between \(m+1\) and \(2m-1\).
Hence it is a summand of \(\mathbf{P}_*\). Adding \(\mathbf{P}_*\) absorbs both
error terms. Adjacent exchanges thus permute the second copy modulo
\(\mathbf{P}_*\) before any nonempty suffix. When \(m=1\), no exchange is
needed and the empty interval of error lengths is irrelevant.

Write \(\bq=q_1\cdots q_n\), retaining repetitions. For \(x=q_n\),
the witness \(\mathbf{f}_x\) ends in \(x\). Move its indexed factor to the
end of the second copy modulo \(\mathbf{P}_*\), with suffix \(x\), and use
\(\mathbf{f}_xx=\mathbf{f}_x\) from \eqref{eq:b:023}. The resulting word is a product of exactly
\(2m\) witnesses containing every index, hence is a summand of
\(\mathbf{P}_*\). This proves \(\mathbf{U}^2q_n\le \mathbf{P}_*\).

For an earlier position set \(x=q_i\) and
\(\mathbf{T}=q_{i+1}\cdots q_n\), a nonempty word. The proof of \eqref{eq:b:034}
uses the original unreduced \(\bq\) and gives \(\mathbf{f}_xx\mathbf{T}=\mathbf{f}_x\mathbf{T}\).
Move the chosen indexed factor to the end of the second copy, apply
this equality, and restore the factor order with suffix \(\mathbf{T}\).
The established exchanges give
\begin{equation}\label{eq:b:058}
\mathbf{P}_*+\mathbf{U}^2x\mathbf{T}=\mathbf{P}_*+\mathbf{U}^2\mathbf{T}.
\end{equation}
Descending induction proves \(\mathbf{U}^2\bq\le \mathbf{P}_*\).

For the left statement reverse \(\bq\) and all \(\mathbf{g}_x\) and apply
this right-side argument. Reversal takes the polynomial of all
allowed index sequences to \(\mathbf{Q}_*\), since reversing a sequence
preserves its length and the occurrence of every index. It takes
the square of the chosen reversed witness product to \(\mathbf{V}^2\).
Reversal invariance of \(E_{\rm band}\) therefore gives
\(\bq\mathbf{V}^2\le \mathbf{Q}_*\). Every term in these two derivations is a term
on \(C\).
\end{proof}

\begin{proposition}[Absorption of the proper selector]\label{band:proper-selector-completeness}
The explicit finite \(\Abar\)-valid list
\(E_{\rm prop}=E_{\rm BG}\cup{\rm GB}\) has this uniform property:
for every nonempty word \(\bq\), with \(C=c(\bq)\), and every finite
nonempty polynomial \(\mathbf{P}\), if \(\bq\le \mathbf{P}\) holds in
\(B_R\times B_L\), then
\begin{equation}\label{eq:b:059}
E_{\rm prop}\vdash \bq^2B_C\bq^2\le \mathbf{P}.
\tag{\text{B-23}}
\end{equation}
In particular, \eqref{eq:b:059} holds whenever \(\bq^2\le \mathbf{P}\) is valid in \(\Abar\).
The left side of \eqref{eq:b:059} equals \(\bq^2\) on proper-generation assignments
and zero on full-\(\Abar\)-generation assignments.
\end{proposition}
\begin{proof}
All 29 identities have been proved \(\Abar\)-valid. Choose
\(\mathbf{U},\mathbf{V},\mathbf{P}_*,\mathbf{Q}_*\) from the preceding lemma and put \(B=B_C\).
Guarded lifting of \eqref{eq:b:057} gives
\begin{equation}\label{eq:b:060}
B\mathbf{U}^2\bq\le B\mathbf{P}_*,\qquad \bq\mathbf{V}^2B\le \mathbf{Q}_*B.
\end{equation}
Using \eqref{eq:b:039}, these inequalities, full-content reflection \eqref{eq:b:041}, the
bounds \eqref{eq:b:056}, and square decrease, respectively, yields
\begin{equation}\label{eq:b:061}
\begin{aligned}
\bq^2B\bq^2
 &=\bq\mathbf{V}^2B\mathbf{U}^2\bq\\
 &\le \mathbf{Q}_*B\mathbf{U}^2\bq\\
 &\le \mathbf{Q}_*B\mathbf{P}_*\\
 &\le \mathbf{Q}_*\mathbf{P}_*\\
 &\le \mathbf{P}^2\\
 &\le \mathbf{P}.
\end{aligned}
\end{equation}
All steps are equational consequences of the fixed finite list.
The varying alphabets, lengths, and witnesses enter only the
constructed terms and finite derivations, not the axiom list.

If \(\bq^2\le \mathbf{P}\) holds in \(\Abar\), it holds in both subsemirings
\(B_R,B_L\), where \eqref{eq:b:023} gives \(\bq^2=\bq\). Thus the cover
hypothesis applies. The assignment-value assertion is \eqref{eq:b:022}.

This proposition supplies only the proper-generation component.
For example, if \(\bq=xyx\) and \(x=a,y=b\), then \(\bq^2=a\) while
\(B_{\{x,y\}}=0\). Removing the guard on full-generation assignments
is therefore not justified by this result. This component alone is
not a finite-basis theorem for \(\Abar\).
\end{proof}

\paragraph{Finite transfer across a word factor.}\label{sec:8}

\paragraph{Ten fixed transfer identities.}
Use distinct ordinary variables \(r,y,s,z\) and the following genuine
nonempty terms:
\begin{equation}\label{eq:t:001}
t=sr,\qquad e=(rys)^2,\qquad \sigma=(ry^2s)^2,
\end{equation}
\begin{equation}\label{eq:t:002}
U=ytytyty,\qquad g=UtU,\qquad J=y+ty+yt.
\end{equation}
The first four transfer identities are
\begin{equation}\label{eq:t:003}
e=eJe+\sigma,
\tag{\text{T-N}}
\end{equation}
\begin{equation}\label{eq:t:004}
eye\leq egege+\sigma,\qquad
egzge\leq z+\sigma,\qquad eU^2e\leq\sigma.
\tag{\text{T-A, R, E}}
\end{equation}
For the remaining six put
\begin{equation}\label{eq:t:005}
A=eye,\quad H=etye,\quad K=eyte,\quad \bv=et,\quad \bu=te,
\end{equation}
\begin{equation}\label{eq:t:006}
E_L=(\bv y\bv)^2,\quad S_L=(\bv y^2\bv)^2,\qquad
E_R=(\bu y\bu)^2,\quad S_R=(\bu y^2\bu)^2.
\end{equation}
Include
\begin{equation}\label{eq:t:007}
H\leq HE_LyE_LH+\sigma,\quad HS_LH\leq\sigma,
\quad HztH\leq z+A+\sigma,
\tag{\text{T-L1--L3}}
\end{equation}
\begin{equation}\label{eq:t:008}
K\leq KE_RyE_RK+\sigma,\quad KS_RK\leq\sigma,
\quad KtzK\leq z+A+\sigma.
\tag{\text{T-R1--R3}}
\end{equation}
All abbreviations are expanded before these are counted as identities.
In particular their list and number of variables are independent of any
word to which the forthcoming transfer theorem will be applied.

\begin{lemma}[Validity of the ten transfer identities]\label{transfer:validity}
All ten identities above are valid in the specified \(\Abar\).
\end{lemma}
\begin{proof}
Every square is in \(\{0,1,a,ab,ba\}\). If \(e\) is a step map, each
\(ewe\) is zero or \(e\); if \(e=0\), all such sandwiches are zero. If
\(e=1\), the product \(rys\) is \(1\), so \(r=y=s=1\). In that case
\(t=U=g=J=1\), and all four initial equations reduce to reflexive
identities or \(z\leq z+1\).

First suppose \(y\) is idempotent, so \(\sigma=e\). For \(e\ne1\), every
sandwich just mentioned is at most \(e\), proving the four initial laws.
The words \(H,K\) are either zero or \(e\), hence idempotent. Sandwiching
by them again yields values at most \(e\), which proves all six endpoint
laws. If \(e=1\), every fixed term in those laws is \(1\), leaving only
\(1\leq1\) or \(z\leq z+1\).

It remains that \(y=b\), so \(\sigma=0\). If \(e=0\), all left sides
vanish. If \(e\ne0\), expanding \((rbs)^2=rb(sr)bs\) and using the table
shows that \(b(sr)b\ne0\), which forces \(t=sr=a\). Therefore
\(U=g=b\), \(J=b+ab+ba=1\), and \(e\in\{a,ab,ba\}\).
The splitting identity \eqref{eq:t:003} is now \(e=e1e\). Also \(ebe\) is \(a\) for \(e=a\)
and zero for \(e=ab,ba\), so \(ebe=ebebe\), proving the first inequality in \eqref{eq:t:004}. For the second inequality in \eqref{eq:t:004}, \(bzb\)
is zero unless \(z=a\), and in that case \(ebe\leq a=z\). Finally
\(U^2=b^2=0\), proving the third inequality in \eqref{eq:t:004}.

For the endpoint laws direct multiplication gives

\begin{center}
\small
\begin{tabular}{cccccc}
\hline
\(e\) & \(A\) & \(H\) & \(K\) & \(E_L\) & \(E_R\) \\
\hline
\(a\) & \(a\) & \(a\) & \(a\) & \(a\) & \(a\) \\
\(ab\) & \(0\) & \(ab\) & \(0\) & \(a\) & \(0\) \\
\(ba\) & \(0\) & \(0\) & \(ba\) & \(0\) & \(a\) \\
\hline
\end{tabular}
\end{center}

Both \(S_L,S_R\) are zero because they contain \(y^2=b^2\). The table
proves the first two inequalities in each of \eqref{eq:t:007} and \eqref{eq:t:008}. For the third inequality in \eqref{eq:t:007}, if \(H=a\), then \(A=a\) absorbs its
left side; if \(H=0\), that side is zero. For \(H=ab\),
\begin{equation}\label{eq:t:009}
HztH=ab\,z\,a\,ab=ab\,z\,ab\leq z.
\end{equation}
Indeed this sandwich is \(ab\) for \(z=ab,1,a\), and zero otherwise.
For the third inequality in \eqref{eq:t:008} the remaining case \(K=ba\) gives
\(KtzK=ba\,a\,z\,ba=ba\,z\,ba\leq z\); its sandwich is \(ba\) for
\(z=ba,1,a\), and zero otherwise. This proves all assignments.
\end{proof}

\begin{lemma}[Transfer of the top-valued component]\label{transfer:top}
Let \(\br,\by,\bs\) now be any three nonempty words and \(\bq=\br\by\bs\). Let \(\mathbf{F}\) be
any term on \(c(\bq)\), and suppose \(\Abar\models \bq^2\leq \mathbf{F}\). The fixed
coordinate identities and the three inequalities in \eqref{eq:t:004} derive
\begin{equation}\label{eq:t:010}
\bq^2\by\bq^2\leq \mathbf{F}+(\br\by^2\bs)^2.
\tag{\text{T-Top}}
\end{equation}
\end{lemma}
\begin{proof}
Retain the transfer abbreviations, now evaluated as words, and introduce a
fresh variable \(X\). The nonempty word
\begin{equation}\label{eq:t:011}
\mathbf{U}_X=\by(\bs\br)\by X\by(\bs\br)\by
\end{equation}
contains \(X\) and every letter of \(\bq\). Apply Lemma~\ref{lem:coordinate-terms}, Coord, to the
valid absorption \(\bq^2\leq \mathbf{F}\) in the frame word \(\mathbf{U}_X\), with designated
letter \(X\) and coordinate \(\mathsf T\). With \(G_X=\mathbf{U}_XX\mathbf{U}_X\), it gives
\begin{equation}\label{eq:t:012}
G_X\bq^2G_X\leq G_X\mathbf{F}G_X+\mathbf{U}_X^2.
\end{equation}
Substitute the nonempty word \(\bs\br\) for \(X\); this gives
\(geg\leq g\mathbf{F}g+U^2\). Multiplying on both sides by \(e\), distributing,
and applying the third inequality in \eqref{eq:t:004} gives
\begin{equation}\label{eq:t:013}
egege\leq eg\mathbf{F}ge+eU^2e\leq eg\mathbf{F}ge+\sigma.
\end{equation}
Apply the first and then the second inequalities in \eqref{eq:t:004}, with its auxiliary variable \(z\) replaced by the
actual term \(\mathbf{F}\), to conclude
\(e\by e\leq \mathbf{F}+\sigma\). Every substitution is a nonempty term. In
particular the chosen factor \(\by\) is allowed to be an arbitrary word.
\end{proof}

\begin{lemma}[Full factor transfer]\label{transfer:full}
For three nonempty words \(\br,\by,\bs\), \(\bq=\br\by\bs\), and any term \(\mathbf{F}\) with
\(c(\mathbf{F})\subseteq c(\bq)\), if \(\Abar\models \bq^2\leq \mathbf{F}\), the fixed finite
theory derives
\begin{equation}\label{eq:t:014}
\bq^2\leq \mathbf{F}+(\br\by^2\bs)^2.
\tag{\text{T-FT}}
\end{equation}
\end{lemma}
\begin{proof}
Use the transfer abbreviations. Top gives \(A\leq \mathbf{F}+\sigma\). Semantically
in \(\Abar\), \(H\leq e\): for step-map \(e\) this is the sandwich property;
for \(e=0\) it is immediate; for \(e=1\), every factor of \(\br\by\bs\) is one,
so \(H=1\). Therefore \(H\leq \mathbf{F}\) is \(\Abar\)-valid, and so is
\begin{equation}\label{eq:t:015}
E_L=(Ht)^2\leq(\mathbf{F}t)^2\leq \mathbf{F}t.
\tag{\text{T-9.1}}
\end{equation}
The equality follows syntactically from \(\bv\by \bv=(et)\by(et)=Ht\); the last
inequality is square decrease. This semantic check selects another input
to the proved Top theorem; it does not assert that \(e\leq \mathbf{F}\) has already
been derived.

The word \(\bv\by \bv\) has the three nonempty factors \(\bv,\by,\bv\), and \(\bv=et\)
contains every letter of \(\bq\). Thus the content hypothesis of Top holds
for its upper term \(\mathbf{F}t\). Apply Top to \eqref{eq:t:015} to get
\(E_L\by E_L\leq \mathbf{F}t+S_L\). Sandwich by \(H\) and use the fixed laws:
\begin{equation}\label{eq:t:016}
\begin{aligned}
H&\leq HE_L\by E_LH+\sigma\leq H\mathbf{F}tH+HS_LH+\sigma\\
 &\leq H\mathbf{F}tH+\sigma\leq \mathbf{F}+A+\sigma\leq \mathbf{F}+\sigma.
\end{aligned}
\tag{\text{T-9.2}}
\end{equation}
Here the third inequality in \eqref{eq:t:007} takes the polynomial substitution \(z=\mathbf{F}\).

Similarly \(K\leq e\) holds in \(\Abar\), so
\(E_R=(tK)^2\leq(t\mathbf{F})^2\leq t\mathbf{F}\) is valid. The word \(\bu\by \bu=tK\) has the
three nonempty factors \(\bu,\by,\bu\) and content \(c(\bq)\). Top yields
\(E_R\by E_R\leq t\mathbf{F}+S_R\). Sandwich by \(K\), and use \eqref{eq:t:008} and the bound
for \(A\), to obtain \(K\leq \mathbf{F}+\sigma\).

Finally expand \eqref{eq:t:003} as \(e=A+H+K+\sigma\). The three proved bounds give
\(e\leq \mathbf{F}+\sigma\), proving FT. All instances use the same ten fixed
identities and the same coordinate theory, regardless of the chosen words.
\end{proof}

\paragraph{Termination and square completeness.}\label{sec:9}

\paragraph{Four two-path identities.}
Recall the scalar notation \(h_T(\bw)=T\bw^2T\), \(\bu\star \bv=(\bu\bv)^2\) and
the alphabet guard \(B_C\) from the scalar section. Use distinct schematic
variables \(x,y,z\); choose each of \(U,V\) independently to be absent or
a fresh context variable. Put \(T=x+y+z\), with every present context
variable also added to this sum. For each of the four choices include
\begin{equation}\label{eq:t:017}
\bigl((h_T(x)\star h_T(y))\star h_T(xUy)\bigr)
       \star h_T(yVx)\leq h_T(xy).
\tag{\text{T-Two-path}}
\end{equation}
Absent contexts are omitted literally. This fixes the parenthesization
and specifies four equations in the original signature after expansion.

\begin{lemma}[Two-path validity and legal localization]\label{termination:two-path}
All four Two-path identities hold in \(\Abar\). Whenever their nonempty
substituted words use only a fixed nonempty alphabet \(C\), they may be
used with their common bound equal to \(T_C=\sum_{x\in C}x\).
\end{lemma}
\begin{proof}
Every displayed argument of \(h_T\) is at most \(T\), by the base
product-to-join inequality. The scalar cases proved in the scalar section
therefore apply. If \(T=a\), each scalar is zero or \(a\). A nonzero left
side forces \(x,y\) to be nonzero idempotents. If their values are the pair
\(\{ab,ba\}\), one of the two path words begins with \(ba\) and ends with
\(ab\). For every possible middle value \(w\), the table gives
\(ba\,w\,ab\in\{0,b\}\); the product \(ba\,ab=b\) in the absent-context
case has the same property. Its square is zero, a contradiction. Otherwise
\(x,y\) lie in one of the nonzero band modes, so \(xy\) is a nonzero
idempotent, and \(h_T(xy)=a\), proving the inequality.

If \(T=1\), all values lie in \(S\), and the square-content calculation
in Lemma~\ref{lem:lifting-semantics} gives
\(h_T(x)\star h_T(y)=(x^2y^2)^2=(xy)^2=h_T(xy)\).
The other two meet factors can only decrease the left side. If \(T=ab\),
nonzero \(h_T(x),h_T(y)\) force \(x=y=ab\); if \(T=ba\), they force
\(x=y=ba\). If \(T=b\) or zero, all scalars vanish. These cases prove
validity.

For localization, every nonempty word on \(C\) is derivably at most
\(T_C\). Substitute the actual nonempty context words where present and
put \(z=T_C\). The displayed join defining \(T\) then equals \(T_C\),
using these already derived bounds. Thus localization is ordinary
substitution followed by equalities, with no new conditional axiom.
\end{proof}

\begin{lemma}[Elimination of all factor errors into the guard]\label{termination:guard}
For every nonempty word \(\bq\), \(C=c(\bq)\), and every term \(\mathbf{F}\) on \(C\)
with \(\Abar\models \bq^2\leq \mathbf{F}\), the fixed finite theory derives
\begin{equation}\label{eq:t:018}
\bq^2\leq \mathbf{F}+\bq^2B_C\bq^2.
\tag{\text{T-Termination}}
\end{equation}
\end{lemma}
\begin{proof}
Choose the following finite list of literal factors of \(\bq^2\). Include
one occurrence of each letter \(x\in C\). For each unordered pair of
distinct letters \(\{x,y\}\), select one occurrence of each in \(\bq\), and
include both forward cyclic paths between these selected occurrences.
They are words \(x\mathbf{U}y,y\mathbf{V}x\), each of length at most \(|\bq|\), and each is
a literal factor of \(\bq^2\). The middle contexts can be absent. Denote
the resulting nonempty list by \(\bv_1,\ldots,\bv_m\); repetitions are harmless.
For a one-letter content it has just its single letter factor.

Start with
\begin{equation}\label{eq:t:019}
\mathbf{W}_0=\bq^2(\bq^2)^m \bq^2=\bq^{2m+4}.
\end{equation}
The \(m\) middle copies of \(\bq^2\) are dedicated blocks. In block \(i\),
replace one literal occurrence of \(\bv_i\) by \(\bv_i^2\), and do this only
once per block, obtaining \(\mathbf{W}_1,\ldots,\mathbf{W}_m\). Each replacement has the
form
\begin{equation}\label{eq:t:020}
\mathbf{W}_{i-1}=\mathbf{R}_i \bv_i \mathbf{S}_i,\qquad \mathbf{W}_i=\mathbf{R}_i \bv_i^2\mathbf{S}_i,
\end{equation}
with \(\mathbf{R}_i,\mathbf{S}_i\) nonempty because the initial and final padding copies
\(\bq^2\) are never touched. At its turn each dedicated block is intact.
Every \(\mathbf{W}_i\) has exactly content \(C\), and \(\mathbf{W}_m\) still contains every
chosen \(\bv_i^2\) as a literal factor, with nonempty contexts.

Square decrease gives \(\mathbf{W}_i\leq \mathbf{W}_{i-1}\leq \mathbf{W}_0\). With \(e=\bq^2\), the
power law gives \(\mathbf{W}_0^2=e\), hence \(\mathbf{W}_i^2\leq e\). Consequently
\(\mathbf{W}_{i-1}^2\leq \mathbf{F}\) is valid in \(\Abar\) for every \(i\). Apply FT to that
valid absorption and to its displayed three-factor decomposition. All
content restrictions hold, and we derive
\begin{equation}\label{eq:t:021}
\mathbf{W}_{i-1}^2\leq \mathbf{F}+\mathbf{W}_i^2.
\end{equation}
Additive transitivity over these finitely many inequalities gives
\begin{equation}\label{eq:t:022}
e\leq \mathbf{F}+E,\qquad E=\mathbf{W}_m^2.
\tag{\text{T-10.1}}
\end{equation}

For each selected factor choose its surviving occurrence in \(\mathbf{W}_m^2\),
writing \(E=\mathbf{A}_i \bv_i^2\mathbf{D}_i\), where \(\mathbf{A}_i,\mathbf{D}_i\) are genuine nonempty words
on \(C\). Their product-to-join bounds yield
\begin{equation}\label{eq:t:023}
E\leq T_C\bv_i^2T_C=h_{T_C}(\bv_i).
\tag{\text{T-10.2}}
\end{equation}
The power law gives \(E^2=E\), and thus \(E\star E=E\). In particular, if
\(E\leq a,b\), monotonicity gives \(E=E\star E\leq a\star b\). This step
does not require \(E\) itself to belong to the scalar lattice.

For distinct \(x,y\), take the four bounds \eqref{eq:t:023} associated to
\(x,y,x\mathbf{U}y,y\mathbf{V}x\), combine them by the parenthesization of Two-path, and use
the localized Two-path law to derive \(E\leq h_{T_C}(xy)\). Interchanging
\(x,y\) and their paths gives the analogous bound for \(yx\). For repeated
pairs, \(h_{T_C}(x^2)=h_{T_C}(x)\) by the power law. Thus \(E\) is below
every scalar factor defining \(B_C\), namely every \(h_{T_C}(x)\) and
every \(h_{T_C}(xy)\). Repeatedly use \(E\star E=E\) to derive
\(E\leq B_C\). Already \(E\leq e\), so
\begin{equation}\label{eq:t:024}
E=E^3\leq eB_Ce.
\tag{\text{T-10.3}}
\end{equation}
Combining \eqref{eq:t:022} and \eqref{eq:t:024} proves Termination. There are finitely many
blocks and derivation steps for each target; the identity list is fixed.
\end{proof}

\begin{lemma}[Squared-word completeness]\label{completeness:square}
For every retained \(\Abar\)-valid absorption \(\bq^2\leq \mathbf{F}\), the fixed
finite theory \(\mathcal E\) derives \(\bq^2\leq \mathbf{F}\).
\end{lemma}
\begin{proof}
The proper-selector theorem proved above gives \(\bq^2B_C\bq^2\leq \mathbf{F}\): in
each band mode \(\bq^2=\bq\), so the hypothesis needed there follows from
the present \(\Abar\)-validity. Termination gives
\(\bq^2\leq \mathbf{F}+\bq^2B_C\bq^2\). Additive transitivity yields the conclusion.
No selector is cancelled, and no error term is transported multiplicatively.
\end{proof}

\end{document}